\documentclass[a4paper, 11pt]{article}
\usepackage{geometry}
\usepackage[utf8]{inputenc}
\usepackage[english]{babel}

\usepackage{authblk}

\usepackage{dsfont} 
\usepackage{amsmath}
\usepackage[justification=centering]{caption}
\usepackage[colorlinks=true, citecolor=red, linkcolor=blue]{hyperref}
\usepackage[nameinlink, capitalize]{cleveref}

\usepackage[colorinlistoftodos]{todonotes}

\usepackage{color}
\usepackage{float}
\usepackage[super]{nth}
\usepackage{booktabs}
\usepackage{amssymb}
\usepackage{amsthm}
\usepackage{rotating}
\usepackage{verbatim}
\usepackage{multirow}
\usepackage{physics}
\usepackage{mathtools}
\usepackage{subcaption}
\usepackage{bm}
\usepackage[toc,page]{appendix}
\usepackage[shortlabels]{enumitem}
\usepackage{nomencl}

\newcounter{mnotecount}[section]

\theoremstyle{plain}
\newtheorem{theorem}{Theorem}
\numberwithin{theorem}{section}
\newtheorem{corollary}{Corollary}
\numberwithin{corollary}{section}
\newtheorem{lemma}{Lemma}
\numberwithin{lemma}{section}
\newtheorem{proposition}{Proposition}
\numberwithin{proposition}{section}

\newtheorem{assumption}{Assumption}

\theoremstyle{definition}

\newtheorem{definition}{Definition}
\numberwithin{definition}{section}
\newtheorem{remark}{Remark}
\numberwithin{remark}{section}

\numberwithin{equation}{section}

\title{The Pauli-Poisson equation and its semiclassical limit}
\author[a]{Jakob Möller}
\affil[a]{Research Platform MMM "Mathematics-Magnetism-Materials" c/o Fak. Mathematik, Univ. Wien, Oskar-Morgenstern-Platz 1, 1090 Vienna, Austria}

\begin{document}

\maketitle


\begin{abstract}
The Pauli-Poisson equation is a semi-relativistic model for charged spin-$1/2$-par-ticles in a strong external magnetic field and a
self-consistent electric potential computed from the Poisson equation in 3
space dimensions. It is a system of two magnetic Schr\"odinger equations
for the two components of the Pauli 2-spinor, representing the two spin states of a fermion, coupled by the additional Stern-Gerlach
term representing the interaction of magnetic field and spin.

We study the global wellposedness in the energy space and the semiclassical  limit of the Pauli-Poisson to the magnetic Vlasov-Poisson equation with Lorentz force and the semiclassical limit of the linear Pauli equation to the magnetic Vlasov equation with Lorentz force.

We use Wigner transforms and a density matrix formulation for mixed states, extending the
work of P. L. Lions \& T. Paul as well as P. Markowich \& N.J. Mauser on the semiclassical limit of the non-relativistic 
Schr\"odinger-Poisson equation.
\end{abstract}

\section{Introduction}

The \emph{Pauli-Poisson equation}, introduced in \cite{moller2023models} similar to the magnetic Schrödinger-Poisson equation (e.g. Barbaroux and Vougalter \cite{barbaroux2016existence, barbaroux2017well}), for a 2-spinor $u^{\hbar,c} = (u_1^{\hbar,c},u_2^{\hbar,c})^T \in (L^2(\mathbb{R}_x^3\times \mathbb{R}_t,\mathbb{C}))^2$ and the self-consistent electric potential $V^{\hbar,c}$ (with given magnetic potential $A$) – which depend on the two parameters  $\hbar$ (Planck constant) and $c$ (speed of light) – is given by the following nonlinear system of 1+1 coupled PDE 
\begin{align} 
    i\hbar\partial_t u^{\hbar,c} &= -\frac{1}{2m}(\hbar \nabla-i\frac{q}{c}A)^2u^{\hbar,c} + qV^{\hbar,c} u^{\hbar,c} - \frac{\hbar q}{2cm} (\sigma \cdot B) u^{\hbar,c},\label{eq:PP_Pauli_unscaled}\\
    -\Delta V^{\hbar,c} &= \rho^{\hbar,c} := |u^{\hbar,c}|^2.\label{eq:PP_Poisson_unscaled}
\end{align}
with initial data
\begin{equation}
    u^{\hbar}(x,0) = u^{\hbar}_I(x)\in (L^2(\mathbb{R}^3))^2.
    \label{eq:PP_data_unscaled}
\end{equation}
and \emph{Pauli current density} defined by
\begin{equation}
 J^{\hbar,c} = \Im(\overline{u^{\hbar,c}}(\hbar \nabla -i \frac{q}{c} A)u^{\hbar,c}) - \hbar \nabla \times (\overline{u^{\hbar,c}}\sigma u^{\hbar,c}).
    \label{eq:Pauli_current}
\end{equation}
Here, $m$ is the mass and $q$ is the charge. 
$A\colon \mathbb{R}^3\rightarrow \mathbb{R}^3$ is the external magnetic vector potential with the corresponding external magnetic field $B = \nabla \times A$, independent of $\hbar$.  

The scalar charge density is defined by  $\rho^{\hbar,c} := |u^{\hbar,c}|^2:= u^{\hbar,c} \overline{u^{\hbar,c}} = |u^{\hbar,c}_1|^2+|u^{\hbar,c}_2|^2$. More generally we write $v\overline{w}$ for the scalar product of two 2-spinors $v=(v_1,v_2)^T$, $w=(w_1,w_2)^T \in \mathbb{C}^2$, i.e.
\begin{equation*}
    {v}\overline{w} := \langle v,w\rangle_{\mathbb{C}^2} = v_1w^*_1 +v_2w^*_2.
\end{equation*}
where $*$ denotes the complex conjugate. The coupling of spin to the magnetic field is given by the Stern-Gerlach term $\sigma \cdot B := \sum_{k=1}^3 \sigma_k B_k$ where the $\sigma_k$ are the Pauli matrices
\begin{align}
    \sigma_1 = \begin{pmatrix}
    0 & 1 \\ 1 & 0
    \end{pmatrix}, && 
     \sigma_2 = \begin{pmatrix}
    0 & -i \\ i & 0
    \end{pmatrix}, &&
     \sigma_3 = \begin{pmatrix}
    1 & 0 \\ 0 & -1
    \end{pmatrix}.
\end{align} 
The expressions $\overline{u^{\hbar,c}}\nabla u^{\hbar,c}$ and $\overline{u^{\hbar,c}}\sigma u^{\hbar,c}$ are to be understood as the 3-vectors with components
\begin{align*}
    (\overline{u^{\hbar,c}_1}, \overline{u^{\hbar,c}_2}) \partial_k \begin{pmatrix} u^{\hbar,c}_1 \\ u^{\hbar,c}_2 \end{pmatrix}, && (\overline{u^{\hbar,c}_1}, \overline{u^{\hbar,c}_2}) \sigma_k \begin{pmatrix} u^{\hbar,c}_1 \\ u^{\hbar,c}_2 \end{pmatrix},
\end{align*}
for $k=1,2,3$. Note that using the Pauli vector identity $(a\cdot \sigma)(b\cdot \sigma) = (a\cdot b)I + i(a \times b) \cdot \sigma$ the Pauli Hamiltonian can be rewritten as 
\begin{equation}
    H = -\frac{1}{2m}(\sigma \cdot (\hbar \nabla-i\frac{q}{c}A))^2 + qV^{\hbar,c}.
\end{equation}
We also use the following notation for the "free" Pauli Hamiltonian
\begin{equation}
    H_0 = -\frac{1}{2m}((\sigma \cdot (\hbar \nabla -i\frac{q}{c}A))^2 .
    \label{eq:P_Hamiltonian}
\end{equation}
The current density can be written as
\begin{equation}
  J^{\hbar,c} =  \Re \left( \overline{u^{\hbar,c}}
  {\sigma}({\sigma} \cdot (- i \hbar \nabla - \frac{q}{c}
  A)) u^{\hbar,c} \right),
  \label{eq:J compact}
\end{equation}
which is to be understood as the 3-vector with components
\begin{equation*}
    J_k^{\hbar,c} = \Re \left( \overline{u^{\hbar,c}}
  \sigma_k({\sigma} \cdot (- i \hbar \nabla - \frac{q}{c}
  A)) u^{\hbar,c} \right) .
\end{equation*}
The key semi-relativistic features of the Pauli equation are the magnetic Laplacian
\begin{equation}
    \frac{1}{2m}(\hbar\nabla-i\frac{q}{c}A)^2,
    \label{eq:magnetic_laplacian}
\end{equation}
and the Stern-Gerlach term  
\begin{equation}
    \frac{\hbar q}{2mc}(\sigma\cdot B)u^{\hbar,c},
    \label{eq:stern_gerlach_term}
\end{equation}
coupling spin and magnetic field $B$. There are two small parameters in \eqref{eq:PP_Pauli_unscaled}-\eqref{eq:PP_Poisson_unscaled}:
\begin{enumerate}[label=(\alph*)]
\item $1/c$ which corresponds to the non-relativistic (or "post-Newtonian"  limit) where the speed of light $c$ is infinite  (i.e. $c \to \infty$) and
\item The Planck constant $\hbar$ which we study in the semiclassical limit $\hbar \to 0$. 
\end{enumerate}
In this asymptotic setting, the Pauli equation  can be seen as either
\begin{itemize}
\item the $O(1/c)$ approximation of the fully relativistic Dirac equation for a 4-spinor describing the two spin states of the particle and its antiparticle, or
\item the $O(1/c)$ correction extending the scalar Schrödinger equation which does not include magnetism nor spin, both being relativistic effects.
\end{itemize}
The two components
of the Pauli equation describe the two spin states of a charged fermion travelling at high speed in the intermediate regime compared to the speed of light, whereas the Poisson equation for the electric potential $V^{\hbar}$
describes the self-interaction with the electric field. The 2-spinor corresponds to the "electron component" of the Dirac equation for a 4-spinor obtained by a Foldy - Wouthuysen transform and projection on the "upper" ("large") component of the 4 spinor (the "lower" ("small") component corresponds to the positron, cf. \cite{bechouche1998semi}). The curl term in \eqref{eq:Pauli_current} arises from the $O(1/c)$ approximation of the Dirac current, cf. \cite{nowakowski1999quantum}. Following \cite{barbaroux2016existence, barbaroux2017well,luhrmann2012mean, michelangeli2015global}, we consider a situation where  the magnetic potential $A$ is externally given and therefore we do not use a superscript $\hbar$ for $A$ and $B$. 

From the point of view of physics this setting is appropriate when the external magnetic field is much stronger than the self-consistent magnetic field. The Pauli-Poisson equation is not consistent in the parameter $1/c$ and thus not a fully self-consistent model. The fully self-consistent semi-relativistic model is given by the Pauli-Poisswell equation defined in \eqref{eq:PPW_Pauli}-\eqref{eq:PPW_current}. 

The Pauli-Poisson equation also arises as the mean field limit where $N\rightarrow \infty$ of the linear $N$-body Pauli equation with Coulomb interaction, introduced in \cite{moller2023models}. 

Since we are interested in the semiclassical  limit  $\hbar \rightarrow 0$ with $c$ fixed we use a scaling of \eqref{eq:PP_Pauli_unscaled}-\eqref{eq:PP_Poisson_unscaled} where the small parameter is a scaled Planck constant, still denoted by $\hbar$. We will also omit the $c$-superscript and write $u^{\hbar}$ instead in order to emphasize the dependence on $\hbar$.

The Pauli-Poisson equation can be generalized to the \emph{Pauli-Hartree equation} (PH)
\begin{align} 
    i\hbar\partial_t \Psi^{\hbar,c} &= -\frac{1}{2}(\hbar \nabla-iA)^2\Psi^{\hbar,c} + V^{\text{ext}} \Psi^{\hbar,c} - \frac{\hbar}{2} (\sigma \cdot B) \Psi^{\hbar,c} + (W \ast|\Psi^{\hbar,c}|^2)\Psi^{\hbar,c},\label{eq:pauli hartree}
\end{align}
We consider the $3d$ case which is physically relevant whenever magnetic fields are considered. In this case only the Pauli-Poisson equation corresponds to the PH equation with
\begin{equation}
    W(x) = \frac{1}{|x|},
\end{equation}
If we neglect spin , i.e. omit the Stern-Gerlach term in the Pauli-Poisson equation and replace the 2-spinor $u^{\hbar}$ by a scalar wave function $\psi^{\hbar}\in L^2(\mathbb{R}^3)$ we obtain the \emph{magnetic Schrödinger-Poisson equation} (m-SP), considered in \cite{barbaroux2016existence, barbaroux2017well}. It is a simplified model compared to the Pauli-Poisson equation since it does not include the description of spin. The m-SP equation is given by
\begin{align}
    i\hbar\partial_t \psi^{\hbar} &= -\frac{1}{2}(\hbar \nabla-iA)^2\psi^{\hbar} + V^{\hbar} \psi^{\hbar}, \label{eq:MSP_Schrödinger}\\
    -\Delta V^{\hbar} &= \rho^{\hbar} :=  |\psi^{\hbar}|^2, \label{eq:MSP_PoissonV}
\end{align}
with initial data
\begin{equation}
    \psi^{\hbar}(x,0) = \psi^{\hbar}_I(x) \in L^2(\mathbb{R}^3).
    \label{eq:MSP_data}
\end{equation}
and \emph{magnetic Schrödinger current density}, given by
\begin{equation}
    J^{\hbar} = \Im (\overline{\psi^{\hbar}}(\hbar \nabla -iA) \psi^{\hbar}).
\end{equation}
More generally we may consider the \emph{magnetic Schrödinger-Hartree equation} (m-SH), studied in \cite{luhrmann2012mean, michelangeli2015global}:
\begin{align} 
    i\hbar\partial_t \psi^{\hbar} &= -\frac{1}{2}(\hbar \nabla-i A)^2\psi^{\hbar} + V^{\text{ext}} \psi^{\hbar} + (W \ast|\psi^{\hbar}|^2)\psi^{\hbar},\label{eq:magnetic schrödinger hartree}
\end{align}
where $V^{\text{ext}}$ is an external potential and $W$ is an interaction kernel depending on $x\in \mathbb{R}^3$. In $3d$ the m-SP equation and m-SH equation with $W(x) \simeq |x|^{-1}$ are equivalent.

Global wellposedness of the m-SP equation has been studied in \cite{barbaroux2017well} for bounded potentials and of the m-SH equation for $A\in L^2_{\text{loc}}$ in \cite{michelangeli2015global}. The mean field limit of the bosonic $N$-body magnetic Schrödinger equation with Coulomb interaction to the m-SH equation \eqref{eq:magnetic schrödinger hartree} was shown by Lührmann in \cite{luhrmann2012mean}.

The non-relativistic limit (i.e. $c \rightarrow \infty$) of the Pauli-Poisson equation \eqref{eq:PP_Pauli_unscaled}-\eqref{eq:PP_data_unscaled} and the m-SP equation \eqref{eq:MSP_Schrödinger}-\eqref{eq:MSP_data} is the \emph{Schrödinger-Poisson equation} (SP) given by
\begin{align}
    i\hbar\partial_t \psi^{\hbar} &= -\frac{\hbar^2}{2} \Delta \psi^{\hbar} + V^{\hbar}\psi^{\hbar}, \label{eq:SP_Schrödinger}\\
    -\Delta V^{\hbar} &= \rho^{\hbar} := |\psi^{\hbar}|^2,
    \label{eq:SP_Poisson}
\end{align}
with initial data
\begin{equation}
    \psi^{\hbar}(x,0) = \psi^{\hbar}_I(x) \in L^2(\mathbb{R}^3),
\end{equation}
and \emph{Schrödinger current density}, given by
\begin{equation}
    J^{\hbar} = \hbar \Im(\overline{\psi ^{\hbar}} \nabla \psi^{\hbar}). \label{eq:SP_current}
\end{equation}
Compare \eqref{eq:SP_current} to \eqref{eq:Pauli_current} where we have additional terms due to the presence of magnetic fields and spin. 

The SP equation is a fully self-consistent $O(1)$ (i.e. non-relativistic) model which arises as the non-relativistic limit of the fully self-consistent \emph{Dirac-Maxwell equation} (DM) of relativistic quantum mechanics, describing the self-interaction  of a charged fermion and its antiparticle with the electromagnetic field, cf. \cite{masmoudi2003nonrelativistic}.

The semiclassical limit of \eqref{eq:SP_Schrödinger}-\eqref{eq:SP_current} to the \emph{Vlasov-Poisson equation} (VP) using Wigner transforms was shown by P.L. Lions \& T. Paul in \cite{lions1993mesures} and by P. Markowich \& N.J. Mauser in \cite{markowich1993classical}  (in $d=1$ the pure state case was dealt with by Zhang, Zheng and Mauser in \cite{zhang2002limit} for appropriate non-unique measure valued weak solutions of the VP equation). Global wellposedness in $H^2(\mathbb{R}^3)$ of \eqref{eq:SP_Schrödinger}-\eqref{eq:SP_Poisson} for mixed states was established in \cite{brezzi1991three, illner1994global} and in $L^2(\mathbb{R}^3)$ in \cite{castella1997l2}. This fits well with the semiclassical limit of the SP equation to the VP equation in $3d$ which works only for particular mixed states satisfying condition \eqref{eq:weight_conditionC} below.

The analysis for the Pauli-Poisson equation is considerably harder than for the Schrö-dinger-Poisson equation. Already for the linear Pauli equation (i.e. where both $V$ and $A$ are given "external" potentials) the analysis is much more complicated than for the (magnetic) Schrödinger equation because of the existence of zero modes due to the presence of the Stern-Gerlach term involving the magnetic field $B=\nabla \times A$, see e.g. \cite{erdHos1997semiclassical}, where even the case of constant magnetic fields shows hard technical challenges. At the core of the technical difficulties, besides the Stern-Gerlach term, is the "advective" term $A\cdot \nabla$ in the magnetic Schrödinger operator.

For the linear Pauli equation with external $V$ and $A$ we include a generalization of Theorem IV.1 from \cite{lions1993mesures}, which states that under the assumption $V\in C^{1,1}$ the semiclassical limit is unique and the Hamiltonian flow of the Vlasov equation is well-defined. For the magnetic Vlasov equation we have a similar result but naturally we have to include additional assumptions on the magnetic potential $A$, stated in Theorem \ref{thm:main} \ref{prop:C1}.

In the fully self-consistent semi-relativistic $O(1/c)$ approximation of the DM equation a magnetostatic, i.e $O(1/c)$ approximation of  Maxwell's  equations is used to self-consistently describe the magnetic field. The magnetic potential $A^{\hbar,c}$ is coupled to $u^{\hbar,c}$ via three Poisson type equations with the Pauli current density as source term. This is the \emph{Pauli-Poisswell equation}, derived in \cite{masmoudi2001selfconsistent}:
\begin{align}
    i\hbar\partial_t u^{\hbar,c} &= -\frac{1}{2}(\hbar \nabla-\frac{i}{c}A^{\hbar,c}
    )^2u^{\hbar,c} + V^{\hbar,c} u^{\hbar,c} -\frac{1}{2} \frac{\hbar}{c} (\sigma \cdot B^{\hbar,c}) u^{\hbar,c}, \label{eq:PPW_Pauli}\\
    -\Delta V^{\hbar,c} &= \rho^{\hbar,c} = -|u^{\hbar,c}|^2, \\
    -\Delta A^{\hbar,c} &= \frac{1}{c}J^{\hbar,c} \label{eq:PPW_PoissonA}
\end{align}
with initial data
\begin{equation}
    u^{\hbar,c}(x,0) = u^{\hbar,c}_I(x) \in (L^2(\mathbb{R}^3))^2.
    \label{eq:ppw_data}
\end{equation}
and Pauli current density, given by
\begin{equation}
    J^{\hbar,c}(u^{\hbar,c},A^{\hbar,c}) = \Im(\overline{u^{\hbar,c}}(\hbar\nabla -\frac{i}{c}A^{\hbar,c})u^{\hbar,c}) -{\hbar}\nabla \times (\overline{u^{\hbar,c}} \sigma u^{\hbar,c}), \label{eq:PPW_current},
\end{equation}
Since $A^{\hbar,c}$ is coupled to $u^{\hbar,c}$ we write a superscript $\hbar,c$ in order to emphasize its dependence on the semiclassical and the relativistic parameter. Compare  \eqref{eq:PPW_current} to \eqref{eq:Pauli_current} and notice that in the former the magnetic potential depends on $\hbar$ and $c$. The semiclassical limit of \eqref{eq:PPW_Pauli}-\eqref{eq:PPW_PoissonA} is work in progress \cite{MaMo23}. The classical limit $c\rightarrow \infty, \hbar \rightarrow 0$ of the DM equation to the VP equation was proven in \cite{mauser2007convergence} where the authors essentially first perform the non-relativistic limit to the SP equation and then the semiclassical limit to the VP equation.  The semiclassical limit of the DM equation to the \emph{relativistic Vlasov-Maxwell} equation is a very hard open problem.  The semiclassical limit of \eqref{eq:PP_Pauli_unscaled}-\eqref{eq:PP_Poisson_unscaled} is a small step towards this goal, being the first true extension of the SP results in \cite{lions1993mesures},\cite{markowich1993classical}. 


\subsection{Wigner transforms and the mixed state Pauli-Poisson equation}
\label{sec:wigner}
Mixed states in quantum mechanics represent a statistical ensemble of possibles states and are the fundamental object since a pure state is a special case of a mixed state. It expresses that, in general, we cannot prepare a quantum system in one precise (pure) state, but rather with a certain probability $\lambda_j$ in one of many possible states $u_j$. The mixed state formulation is necessary from a technical point of view when dealing with the semiclassical limit of the SP and Pauli-Poisson equations since uniform $L^2$ estimates for the Wigner transform are only possible in a mixed state formulation, cf. Assumption \ref{thm:remark_pure_states}. A mixed state is represented by the density matrix which is defined as follows.

Let $\{u^{\hbar}_j\}_{j\in \mathbb{N}}$, $u^{\hbar}_j = (u_{j,1}^{\hbar},u_{j,2}^{\hbar})^T$ be an orthonormal system in $(L^2(\mathbb{R}^3))^2$. We define the density matrix $\rho^{\hbar}$
\begin{align} \label{eq:Def_rho}
    \rho^{\hbar}(x,y,t) &:=  \sum_{j=1}^{\infty} \lambda^{\hbar}_j u^{\hbar}_j(x,t) \overline{u^{\hbar}_j(y,t)} \\ &= \sum_{j=1}^{\infty} \lambda^{\hbar}_j \left(u_{j,1}^{\hbar} (x,t) {u_{j,1}^{\hbar}(y,t)^*} +u_{j,2}^{\hbar} (x,t) {u_{j,2}^{\hbar}(y,t)^*}\right) \nonumber ,
\end{align}
where  $\lambda = \{\lambda_j^{\hbar}\}_{j\in \mathbb{N}}$ is a normally convergent series such that $\lambda_j^{\hbar} \geq 0$ and $\sum_j \lambda_j^{\hbar} = 1$. The $\lambda_j^{\hbar}$'s are the occupation probabilities  of the states $u_j^{\hbar}$. If  there is a $k$ such that $\lambda_j^{\hbar} = 1$ for $j=k$ and $\lambda_j^{\hbar} = 0$ otherwise then $\rho^{\hbar}$ represents a \emph{pure state}. Otherwise it represents a \emph{mixed state}. The density matrix $\rho^{\hbar}$ can be considered as the kernel of a Hilbert-Schmidt, hermitian, positive and trace class operator $\varrho^{\hbar}$ on $L^2(\mathbb{R}^3)$, called \emph{density operator}. As usual in physics, we will often identify $\rho^{\hbar}$ and $\varrho^{\hbar}$. The diagonal of $\rho^{\hbar}(x,y)$ corresponds to the mixed state particle density and is defined by
\begin{equation}
    \rho^{\hbar}_{\text{diag}}(x) := \rho^{\hbar}(x,x) = \sum_{j=1}^{\infty} \lambda_j^{\hbar} |u_j^{\hbar}(x)|^2 \in L^1(\mathbb{R}^3_x).
    \label{eq:rho_diag}
\end{equation}
Note that since $u^{\hbar}_j\in (L^2(\mathbb{R}^3))^2$ and since $\sum_j \lambda^{\hbar}_j$ converges normally this expression is well-defined. Indeed, consider
\begin{equation*}
    \rho^{\hbar}(x+z,x) = \sum_{j=1}^{\infty} \lambda_j^{\hbar} u^{\hbar}_j(x+z)\overline{u^{\hbar}_j(x)}.
\end{equation*}
We can then let $z\rightarrow 0$ since $\rho^{\hbar}(x+z,x) \in C(\mathbb{R}_z^3,L^1(\mathbb{R}^3_x))$.
The time evolution of $\rho^{\hbar}$ is given by the von Neumann equation:
\begin{equation}
    i\hbar \frac{\partial \rho^{\hbar}}{\partial t} = [H,\rho^{\hbar}].
    \label{eq:von_Neumann_Liouville}
\end{equation}
The scalar Wigner transform $f^{\hbar}(x,\xi,t)$ of $\rho^{\hbar}$ is defined by
\begin{equation}
    f^{\hbar}(x,\xi,t) := \frac{1}{(2\pi\hbar)^{3}} \int_{\mathbb{R}^3_y} e^{-i\xi\cdot y} \rho^{\hbar}(x+\frac{\hbar y}{2}, x-\frac{\hbar y}{2},t) \dd y.
    \label{eq:WT_rho}
\end{equation}
Note that some authors use the opposite sign in the exponential or a different normalization. 
More generally, we define the \emph{Wigner matrix} $F^{\hbar}$ (cf. \cite{gerard1997homogenization}) as the Wigner transform of a matrix valued density matrix $R^{\hbar}$, i.e.
\begin{equation}
    F^{\hbar}(x,\xi,t) = \frac{1}{(2\pi\hbar)^3} \int_{\mathbb{R}_y^3} e^{-i\xi \cdot y} R^{\hbar}(x+\frac{\hbar y}{2}, x-\frac{\hbar y}{2},t) \dd y,
    \label{eq:wigner_matrix}
\end{equation}
where
\begin{equation*}
    R^{\hbar}(x,y,t) := \sum_{j=1}^{\infty} \lambda_j^{\hbar} u_j^{\hbar}(x,t)\otimes \overline{u_j^{\hbar}(y,t)},
\end{equation*}
where $\otimes$ denotes the tensor product of vectors. Then we can define the scalar Wigner transform as $f^{\hbar}=\Tr(F^{\hbar})$ and the scalar density matrix as $\rho^{\hbar} = \Tr(R^{\hbar})$ where $\Tr$ denotes the $2\times 2$ matrix trace. Similarly we can define $R^{\hbar}_{\text{diag}}$ as 
\begin{equation}
    R^{\hbar}_{\text{diag}}(x) = R^{\hbar}(x,x).
\end{equation}
A simple calculation shows that
\begin{equation}
    \rho^{\hbar}_{\text{diag}}(x) = \int_{\mathbb{R}^3_{\xi}} f^{\hbar}(x,\xi) \dd \xi.
    \label{eq:density_wigner}
\end{equation}
and
\begin{equation}
    R^{\hbar}_{\text{diag}}(x) = \int_{\mathbb{R}^3_{\xi}} F^{\hbar}(x,\xi) \dd \xi.
    \label{eq:density_wigner_matrix}
\end{equation}
Since $\{u_j^{\hbar}\}_{j\in \mathbb{N}}$ is a bounded family in $(L^2(\mathbb{R}^3))^2$, $f^{\hbar}$ (resp. $F^{\hbar}$) is a bounded family in $\mathcal{S}'(\mathbb{R}^3_x \times \mathbb{R}^3_{\xi})$ (resp. $(\mathcal{S}'(\mathbb{R}^3_x \times \mathbb{R}^3_{\xi}))^{2\times 2}$) (cf. \cite{gerard1997homogenization}, Proposition 1.1).
The Wigner transform  $f^{\hbar}$ (resp. $F^{\hbar}$) is oscillatory in general and attains negative values. Thus it can only be interpreted as a \emph{quasi probability density}. 

Since $f^{\hbar}$ (resp. $F^{\hbar}$) is bounded in $S'(\mathbb{R}^3_x \times \mathbb{R}^3_{\xi})$ (resp. $(S'(\mathbb{R}^3_x \times \mathbb{R}^3_{\xi}))^{2 \times 2}$) there is a subsequence $\{\hbar_k\}$ going to $0$ such that $f^{\hbar_k}$ (resp. $F^{\hbar_k}$ ) converges to a non-unique limit $f$ (resp. $F$) (cf. \cite{gerard1997homogenization}). It can be shown that $f$ (resp. $F$) is a nonnegative Radon measure, (resp. matrix valued Radon measure), i.e. for all $z\in \mathbb{C}^2$,
\begin{equation*}
\sum_{i,j} F_{ij} z_i \overline{z_j} \geq 0.
\end{equation*}
One possible way of showing the non-negativity of $f$ is to define the Husimi function $\tilde{f}^{\hbar}$ (resp. $\tilde{F}^{\hbar}$)
\begin{equation}
    \label{eq:Husimi}
    \tilde{f}^{\hbar}(x,\xi) := f^{\hbar}(x,\xi) \ast_x G^{\hbar}(x) \ast_{\xi} G^{\hbar}(\xi), \quad G^{\hbar}(z) := \frac{1}{(\pi\hbar)^{3/2}}e^{-|z|^2/\hbar}.
\end{equation}
One can easily see that $\tilde{f}^{\hbar}$ defines a pointwise nonnegative function since the variance of the Gaussian is chosen to be $\hbar$. This is related to the uncertainty principle which states that the product of the variances $\sigma_x$ and $\sigma_{\xi}$ of $x$ and $\xi$ is always greater than $\hbar/2$, i.e. $\sigma_x \sigma_{\xi} \geq \hbar/2$. 

One observes that the accumulation points of $f^{\hbar}$ (resp. $F^{\hbar}$) are the accumulation points of $\tilde{f}^{\hbar}$ (resp. $\tilde{F}^{\hbar}$). One can alternatively prove the nonnegativity of $f$ or $F$ without using the Husimi function by an argument using the Bochner-Schwartz theorem or by an argument using coherent states (cf. \cite{gerard1997homogenization}). $f$ is called \emph{Wigner measure} and $F$ is called \emph{Wigner matrix measure}, (cf. \cite{lions1993mesures, gerard1997homogenization}). It is related to the semiclassical measures in \cite{gerard1991mesures},\cite{gerard1997homogenization} and measures the loss of compactness in $L^2$ in \cite{gerard1991microlocal}. For details on the convergence of Wigner transforms and the Husimi function we refer to Appendix \ref{sec:wigner measures}. 

Note that one can also define the following test function space (in fact an algebra) $\mathcal{A}$ which is more adapted to the Wigner transform $f^{\hbar}$,
\begin{equation}
    \label{algebra A}
    \mathcal{A} := \{\phi \in C_0(\mathbb{R}^3_x \times \mathbb{R}^3_{\xi}) \colon \mathcal{F}_{\xi}[\phi](x,\eta) \in L^1(\mathbb{R}^3_{\eta}, C_0(\mathbb{R}^3_x ))\}
\end{equation}
Under the same assumptions one can show that $f^{\hbar}$ and $F^{\hbar}$ have a weak limit in $\mathcal{A}'$. However, this technicality will make no difference to our analysis and we will either use $\mathcal{S}'$ or $\mathcal{A}'$. The algebra $\mathcal{A}$ was introduced in \cite{lions1993mesures} and used e.g. in \cite{zhang2002limit}.
 \\

Since we have to work with mixed states we formulate a mixed state version of the Pauli-Poisson equation \eqref{eq:PP_Poisson_unscaled}-\eqref{eq:Pauli_current}. Let $\{u_j^{\hbar}\}$ be an $(L^2(\mathbb{R}^3))^2$-orthonormal system. Then the  \emph{mixed state Pauli-Poisson equation} is given by
\begin{align}
    i\hbar\partial_t u_j^{\hbar} &= -\frac{1}{2}(\hbar \nabla-iA)^2u_j^{\hbar} + V^{\hbar} u_j^{\hbar} -\frac{1}{2} \hbar (\sigma \cdot B) u_j^{\hbar}, \label{eq:PP_Pauli_mixed}\\
    -\Delta V^{\hbar} &= \sum_{j=1}^{\infty} \lambda_j^{\hbar} |u_j^{\hbar}|^2 = \rho^{\hbar}_{\text{diag}}, \\
    u_j^{\hbar}(x,0) &= u^{\hbar}_{j,I}(x) \in (L^2(\mathbb{R}^3))^2.
    \label{eq:PP_data_mixed}
\end{align}
where the mixed state Pauli current density is given by
\begin{equation}
    J^{\hbar}(u^{\hbar},A) = \sum_{j=1}^{\infty} \lambda_j^{\hbar}\left[\Im(\overline{u_j^{\hbar}}({\hbar}\nabla -{i}A)u_j^{\hbar}) -{\hbar}\nabla \times (\overline{u_j^{\hbar}} \sigma u_j^{\hbar})\right], \label{eq:PPW_current_mixed}
\end{equation}
or, similar to \eqref{eq:J compact}
\begin{equation}
 J^{\hbar} =  \sum_{j=1}^{\infty} \lambda_j^{\hbar} \Re \left( \overline{u^{\hbar}_j}
  {\sigma}({\sigma} \cdot (- i \hbar \nabla -
  A)) u_j^{\hbar} \right),
\end{equation}
It can be written as a first order moment of the Wigner transform, 
\begin{equation}
    J^{\hbar} = \int_{\mathbb{R}^3_{\xi}} \Tr(\sigma (\sigma \cdot(\xi-A(x))F^{\hbar}(x,\xi))) \dd \xi
\end{equation}
The continuity equation holds for the mixed densities:
\begin{equation}
    \partial_t \rho^{\hbar}_{\text{diag}} + \text{div}_x  J^{\hbar} = 0.
\end{equation}
Rewriting the mixed state Pauli-Poisson equation \eqref{eq:PP_Pauli_mixed}-\eqref{eq:PP_data_mixed} in the density matrix formulation using the von Neumann equation and taking its Wigner transform one obtains the \emph{Pauli-Wigner-Poisson equation} for $F^{\hbar}$,
\begin{align}
\begin{split}
    \partial_t F^{\hbar} + \xi \cdot \nabla_x F^{\hbar} - \mathcal{F}_y[\beta[A]]\ast_{\xi} \nabla_x F^{\hbar} - \theta[A] (\xi F^{\hbar}) + \frac{1}{2}\theta[|A|^2]F^{\hbar}& \\ - \frac{\hbar}{2}\theta[\sigma \cdot B]F^{\hbar} +\theta[V^{\hbar}]F^{\hbar}&= 0, \label{eq:pauli_wigner}
\end{split}\\
-\Delta V^{\hbar} = \rho_{\text{diag}}^{\hbar} &= \int_{\mathbb{R}^3_{\xi}} f^{\hbar} \dd \xi, \\
F^{\hbar}(t=0) &= F^{\hbar}_I,
\label{eq:pauli wigner data}
\end{align}
where $\theta[\cdot]$ is the pseudo-differential operator defined by
\begin{equation}
    (\theta[\cdot]\Phi^{\hbar})(x,\xi,t) := \frac{1}{(2\pi)^3}\int_{\mathbb{R}^6} \delta[\cdot](x,y,t)\Phi^{\hbar}(x,\eta,t) e^{-i(\xi-\eta)\cdot y} \dd \eta \dd y. 
    \label{eq:PDO}
\end{equation}
where
\begin{equation}
    \beta[g] := \frac{1}{2}(g(x+\frac{\hbar y}{2})+g(x-\frac{\hbar y}{2})),
    \label{eq:beta}
\end{equation}
and 
\begin{equation}
    \delta[g] := \frac{i}{\hbar}(g(x+\frac{\hbar y}{2})-g(x-\frac{\hbar y}{2})).
    \label{eq:delta}
\end{equation}
The initial Wigner matrix $F_I^{\hbar}$ is the Wigner transform of the initial matrix valued density matrix $R^{\hbar}_I$, i.e.
\begin{align}
    F^{\hbar}_I &= \frac{1}{(2\pi\hbar)^{3}} \int_{\mathbb{R}^3} e^{-i\xi\cdot y} R^{\hbar}_I(x+\frac{\hbar y}{2}, x-\frac{\hbar y}{2}) \dd y, \\ 
        R^{\hbar}_I(x,y) &= \sum_{j=1}^{\infty} \lambda_j^{\hbar} u_{j,I}^{\hbar}(x) \otimes \overline{u_{j,I}^{\hbar}(y)}.
\end{align}
The energy of the Pauli-Wigner equation is given by
\begin{equation}
      E(t) := \tr(H_0 R^{\hbar}(t)) + \iint_{\mathbb{R}_x^3 \times \mathbb{R}_y^3} \frac{\rho^{\hbar}_{\text{{diag}}}(x,t) \rho_{\text{{diag}}}^{\hbar}(y,t)}{|x-y|} \dd x \dd y.
      \label{eq:energy_pauli_wigner}
\end{equation}
In fact the second term on the RHS is equivalent to
\begin{equation*}
    \int |\nabla V^{\hbar}(x,t)|^2 \dd x.
\end{equation*}

\begin{remark}
We will use the important observation that a particular bound on the occupation probabilities $\lambda_j^{\hbar}$ in Assumption \ref{thm:remark_pure_states} (which excludes a pure state formulation!) leads to uniform in $\hbar$ bounds on the $L^2$-norm of the Wigner matrix $F$. In order to pass to the limit we also need uniform estimates for the density $\rho^{\hbar}_{\text{diag}}$ which were obtained in \cite{lions1993mesures} using Lieb-Thirring estimates. In the case of the Pauli equation we have to use magnetic Lieb-Thirring estimates adapted to the Pauli Hamiltonian from \cite{shen1998moments}. The analysis of the Pauli equation is considerably harder than for the magnetic Schrödinger equation without spin due to the presence of the magnetic field in the Stern-Gerlach term $\sigma\cdot B$ and the existence of zero modes (i.e. non-trivial eigenstates with eigenvalue zero), cf. \cite{erdHos1997semiclassical}.
\end{remark}

\subsection{Main result}

In the following we will give a rigorous proof by a nontrivial extension of the Wigner measure analysis in \cite{gerard1997homogenization},\cite{lions1993mesures},\cite{markowich1993classical} from the scalar case of a simple Schr\"odinger equation to the 2-spinor case of the Pauli equation.
The density matrix formulation introduced in the section above is necessary because for the nonlinear case $V^{\hbar} = |x|^{-1} \ast \rho^{\hbar}$ uniform $L^2$-estimates for the Wigner transform are only possible in this setting. Indeed, $\|f^{\hbar}\|_2 \leq C$ implies that $\hbar^{-3} \tr(\rho_{\hbar}^2) \leq C$ (use \eqref{eq:WT_rho}, Plancherel and a change of variables). Together with the requirements $\tr(\rho^{\hbar})=1$ and $\lambda_j^{\hbar} \geq 0$ this implies that an infinite number of the $\lambda_j^{\hbar}$'s has to be different from zero. Therefore we make the following assumption. 

\begin{assumption}
\label{thm:remark_pure_states}
Let $R^{\hbar}$ or $\rho^{\hbar}$ be a matrix valued density matrix or density matrix defined by an orthonormal system $\{u_j^{\hbar}\}\subset (L^2(\mathbb{R}^3))^2$ and occupation probabilities $\lambda^{\hbar}_j \in [0,1]$. We assume that
\begin{align}
    \lambda_j^{\hbar} \geq 0, \quad \sum_{j=1}^{\infty} \lambda^{\hbar}_j = 1 ,
\end{align}
\begin{equation}
    \frac{1}{\hbar^3}\sum_{j=1}^{\infty} (\lambda^{\hbar}_j)^2 = \frac{1}{\hbar^3} \|\lambda^{\hbar}\|_2^2\leq C. \label{eq:weight_conditionC}
\end{equation}
Since \eqref{eq:weight_conditionC} implies that the sequence $\{\lambda^{\hbar}\}$ depends on $\hbar$ the reason for the superscript becomes apparent.
\end{assumption}

We have the following main result for the semiclassical limit of the linear Pauli equation and the Pauli-Poisson equation.
\begin{theorem}
\label{thm:main} \textbf{\emph{Semiclassical limit of linear Pauli and Pauli-Poisson}}

Let $\{u^{\hbar}_j\}_{j\in \mathbb{N}}\in C(\mathbb{R}_t,(L^2(\mathbb{R}_x^3))^2)$ be a solution of the mixed state Pauli-Poisson equation \eqref{eq:PP_Pauli_mixed}-\eqref{eq:PP_data_mixed} with associated matrix valued density matrix $R^{\hbar}$ such that the occuptation probabilities satisfy Assumption \ref{thm:remark_pure_states}. Let $F^{\hbar}$ be the associated Wigner matrix solving the Pauli-Wigner equation \eqref{eq:pauli_wigner} with initial data $F^{\hbar}_I(x,p) = F^{\hbar}(x,p,0)$. Assume that $F^{\hbar}_I$ converges up to a subsequence in $\mathcal{S}'(\mathbb{R}^3_x \times \mathbb{R}^3_p)^{2\times 2}$ to a nonnegative matrix-valued Radon measure $F_I$. Let $p=\xi-A(x)$.
\begin{enumerate}[label=(\roman*)]
    \item \label{prop:C1} \emph{\textbf{Linear Pauli equation.}} Let $A,V\in C^1(\mathbb{R}^3)$ such that $B :=\nabla \times A\in C(\mathbb{R}^3)$. 
Then $F^{\hbar}$ converges weakly* up to a subsequence in $(\mathcal{S}')^{2\times 2}$ to $F\in C_b(\mathbb{R}_t,\mathcal{M}^{2\times 2}_{w*})$ such that $F$ solves the \textbf{\emph{magnetic Vlasov equation with Lorentz force}}
\begin{equation}
     \partial_t F + p \cdot \nabla_x F +(-\nabla_x V + p\times B)\cdot \nabla_p F = 0,
     \label{eq:limit vlasov}
\end{equation}
in $(\mathcal{D}')^{2\times 2}$ verifying the initial condition
\begin{equation}
    F(x,p,0) = F_I(x,p) \quad \text{in }\mathbb{R}_x^3\times \mathbb{R}_p^3.
    \label{eq:limit vlasov data}
\end{equation}
    If additionally, $A,V \in C^{1,1}(\mathbb{R}^3)$ and there exist constants $C_1,C_2>0$ such that
\begin{equation}
    V(x) > -C(1+|x|^2) \quad \text{for all } x\in \mathbb{R}^3
    \label{eq:V bounded from below}
\end{equation}
\begin{equation}
    |A(t)| \lesssim e^{C_2t} \quad \text{for all } t
\end{equation}
then $f=\Tr(F)$ is the unique solution  in $C_b(\mathbb{R}_t,\mathcal{M}_{w*})$ of the scalar version of \eqref{eq:limit vlasov}-\eqref{eq:limit vlasov data}and $f$ is given by the transport of $f_I$ by the Hamiltonian flow
\begin{equation}
    \dot{x} = p, \quad \dot{p} = E+p\times B.
    \label{eq:hamiltonian equations}
\end{equation}
where $E:= -\nabla_x V$.
\item \label{thm:semiclassical_limit_nonlinear} \emph{\textbf{Pauli-Poisson equation.}}
Let $V^{\hbar}$ be given by $-\Delta V^{\hbar}= \rho^{\hbar}_{\text{\emph{diag}}}$ and suppose $A \in W^{1,7/2}(\mathbb{R}^3)$. 
Moreover, suppose that $\{F_I^{\hbar}\}$ is a bounded sequence in $(L^2(\mathbb{R}^6))^{2\times 2}$ and that the initial energy $E(0)$ is bounded independently of $\hbar$. Then $F^{\hbar}$ converges weakly* up to a subsequence in $L^{\infty}(I,L^2(\mathbb{R}_x^3 \times \mathbb{R}_{p}^3)^{2\times 2})$ to
\begin{equation*}
    F \in C_b(\mathbb{R}_t,\mathcal{M}^{2\times 2}_{w*})\cap L^{\infty}(I,L^1\cap L^2(\mathbb{R}_x^3 \times \mathbb{R}^3_{p})^{2\times 2})
\end{equation*}
such that $F$ solves  the \textbf{\emph{Vlasov-Poisson equation with Lorentz force}}
\begin{equation}
     \partial_t F + p \cdot \nabla_x F +(-\nabla_x V + p\times B)\cdot\nabla_p F = 0,
     \label{eq:limit vlasov poisson vlasov}
\end{equation}
in $(\mathcal{D}')^{2\times 2}$ and
\begin{equation}
    -\Delta V(x,t) = \rho_{\text{\emph{diag}}}(x,t), \quad \rho_{\text{\emph{diag}}}(x,t) = \int_{\mathbb{R}^3_{\xi}} f(x,p,t) \dd p
    \label{eq:limit vlasov poisson poisson}
\end{equation}
where $f=\Tr(F)$, verifying the initial condition
\begin{equation}
    F(x,p,0) = F_I(x,p) \quad \text{in }\mathbb{R}_x^3\times \mathbb{R}_p^3.
    \label{eq:limit vlasov poisson data}
\end{equation}

\item \label{thm_Pauli_current} \emph{\textbf{Pauli current density.}}
Let $A \in W^{1,\frac{7}{2}}(\mathbb{R}^3)$. The mixed state Pauli current density $J^{\hbar}$ defined by
\begin{equation}
    J^{\hbar} = \sum_{j=1}^{\infty} \lambda_j^{\hbar} \left[ \Im(\overline{u_j^{\hbar}}(\hbar \nabla -iA)u_j^{\hbar}) - \hbar \nabla \times (\overline{u_j^{\hbar}}\sigma u_j^{\hbar}) \right]
\end{equation}
converges in $\mathcal{D}'$ to
\begin{equation}
    J = \int_{\mathbb{R}^3_p} pf \dd p.
\end{equation}
\end{enumerate}
\end{theorem}

We will prove Theorem \ref{thm:main} in Section \ref{sec:semiclassical_limit}. We have the following global wellposedness result for the mixed state Pauli-Poisson equation \eqref{eq:PP_Pauli_mixed}-\eqref{eq:PP_data_mixed}. Here,  $\mathbf{u} := \{u_j\}_{j\in \mathbb{N}}$ and $\mathbf{u}_{I} = \{u_{j,I}\}_{j\in \mathbb{N}}$. The energy space $\mathcal{H}^1(\mathbb{R}^3)$ will be defined in Section \ref{sec:GWP}.

\begin{theorem}
\label{thm:PP_gwp} \textbf{\emph{Global wellposedness of Pauli-Poisson}}

Let $A\in L^2_{\text{\emph{loc}}}(\mathbb{R}^3)$, $|B| \in L^{2}(\mathbb{R}^3)$. For any $u_I \in \mathcal{H}^1(\mathbb{R}^3)$ there exists a unique solution to the initial value problem \eqref{eq:PP_Pauli_mixed}-\eqref{eq:PP_data_mixed} in $C(\mathbb{R},\mathcal{H}^1(\mathbb{R}^3))\cap C^1(\mathbb{R},\mathcal{H}^1(\mathbb{R}^3)^*)$. In particular, the solution is in $(L^2(\mathbb{R}^3))^2$ for all times. If $\mathbf{u}_{n,I},\mathbf{u}_I \in \mathcal{H}^1(\mathbb{R}^3)$ are initial data satisfying $ \mathbf{u}_{n,I}\rightarrow\mathbf{u}_I$ in $\mathcal{H}^1(\mathbb{R}^3)$ with corresponding unique  solutions $\mathbf{u}_n\in C(\mathbb{R},\mathcal{H}^1)\cap C^1(\mathbb{R},\mathcal{H}^{1*})$ and  $\mathbf{u}\in C(\mathbb{R},\mathcal{H}^1)\cap C^1(\mathbb{R},\mathcal{H}^{1*})$ then $\mathbf{u}_n \rightarrow \mathbf{u}$ in $L^{\infty}(\mathbb{R},\mathcal{H}^1(\mathbb{R}^3))$.
\end{theorem}

One can generalize the global wellposedness of the Pauli-Poisson equation to the PH equation \eqref{eq:pauli hartree} with slight changes in the proof (cf. \cite{michelangeli2015global}).

\begin{corollary}[Global wellposedness of PH]
\label{thm:PH_gwp}
Under the assumptions of Theorem \ref{thm:PP_gwp} and assuming that $W$ is even, $W \in L^{r_1}(\mathbb{R}^3) + L^{\infty}(\mathbb{R}^3)$ for $3/2 \leq r_1 \leq \infty$ and $\nabla W \in L^{r_2} + L^{\infty}$ for $1 \leq r_2 \leq \infty$ the PH equation is globally wellposed in $\mathcal{H}^1(\mathbb{R}^3)$.
\end{corollary}

\begin{remark}
The question arises whether the PH equation can be posed in arbitrary space dimensions. The three dimensional magnetic field $B=\nabla \times A$ has to be replaced by its $d$-dimensional generalization $\nabla \wedge A$. In this case, following the result for the m-SH equation \cite{michelangeli2015global}, the conditions for $W$ would be: $W$ even, $W \in L^{r_1}(\mathbb{R}^3) + L^{\infty}(\mathbb{R}^3)$ for $\max\{1,d/2\} \leq r_1 \leq \infty$ ($r_1 > 1$ if $d=2$) and $\nabla W \in L^{r_2} + L^{\infty}$ for $\max\{1,d/3\} \leq r_2 \leq \infty$.
\end{remark}

\begin{remark}
Global wellposedness of the VP equation \eqref{eq:limit vlasov poisson vlasov}-\eqref{eq:limit vlasov poisson data} without Lorentz force (i.e. $B=0$) data was established in \cite{pfaffelmoser1992global}.  The limit of strong magnetic fields of the VP equation with Lorentz force was established  in \cite{golse1999vlasov} and the homogenization limit in \cite{frenod1998homogenization}. Global weak solutions of the relativistic Vlasov-Maxwell equation were established in  \cite{diperna1989global}.
\end{remark}

\begin{remark}
If one considers small time scales only one can alternatively use WKB methods where a special ansatz for the form of the wave function is chosen.  The semiclassical limit of the SP equation to the Euler-Poisson equation using WKB analysis was shown by Zhang in \cite{zhang2002wigner, zhang2008wigner} and by  Alazard and Carles in \cite{alazard2007semi}. Grenier \cite{grenier1998semiclassical} proved the semiclassical limit of the cubic NLS and recently, Gui and Zhang \cite{gui2022semiclassical} proved the semiclassical limit of the Gross-Pitaevskii equation.  The semiclassical limit of the Pauli-Poisswell equation \eqref{eq:PPW_Pauli}-\eqref{eq:PPW_current} to the Euler-Poisswell equation using WKB analysis was shown in \cite{yang2023semi}.
\end{remark}

In the case of the linear Pauli equation with sufficiently smooth external electromagnetic potentials we can use Theorem 6.1 from  \cite{gerard1997homogenization} in order to heuristically determine the semiclassical limit equation of the Pauli-Poisson equation \eqref{eq:PP_Pauli_unscaled}-\eqref{eq:PP_data_unscaled}. Write the Pauli equation as 
\begin{align}
    \hbar \partial_t u^{\hbar} + P(x,\hbar D) u^{\hbar} &= 0, \quad x\in \mathbb{R}^3_x, t\in \mathbb{R} \\
    u^{\hbar}(x,0) &= u^{\hbar}_{I}(x) \in (L^2(\mathbb{R}^3))^2.
\end{align}
where $P(x,\hbar D)$ is the pseudo-differential operator with matrix-valued symbol given by
\begin{equation*}
    P(x,\xi) := \frac{i}{2} (\sigma \cdot (\xi -A(x)))^2 +i V\text{Id}.
\end{equation*}
The $2\times 2$-matrix $-iP(x,\xi)$ has one eigenvalue with multiplicity $2$, given by
\begin{equation}
    \lambda(x,\xi) =  \frac{1}{2}\|\xi-A(x)\|^2 + V(x),
    \label{eq:eigenvalue_pauli_symbol}
\end{equation}
where $\|\xi-A\|^2 = \sum_{j=1}^3 (\xi_j-A_j)^2$ and we can deduce that $F^{\hbar}$ defined in \eqref{eq:WT_rho} converges to $F$ solving
\begin{equation}
    \partial_t F + \{\lambda,F\} = 0.
    \label{eq:wigner_equation_from_gmmp}
\end{equation} 
where $\{h,g\} = \nabla_{\xi} h\cdot\nabla_x g - \nabla_{x} h\cdot\nabla_{\xi} g $ denotes the Poisson bracket.
Using \eqref{eq:eigenvalue_pauli_symbol} we have
\begin{equation}
    \partial_t F + \xi\cdot \nabla_x F + (\nabla_x A)\xi \cdot \nabla_{\xi}F - (\nabla_x A) A\cdot \nabla_{\xi} F - A\cdot \nabla_x F - \nabla_x V(x) \cdot\nabla_{\xi} F = 0 
    \label{eq:VLF_beforeCOV}
\end{equation}
We easily obtain after a change of variables $p = \xi -A$ and by using $B = \nabla \times A$, 
\begin{equation}
    \partial_t F + p \cdot \nabla_x F +(-\nabla_x V + p\times B)\nabla_p F = 0 \label{eq:Liouville_Lorentz},
\end{equation}
which is the Vlasov equation with Lorentz force for an electron. Note that for an electron with negative energy (a positron) there would be a minus sign in front of the Lorentz force term. 

\begin{remark}
For the linear Dirac equation it was shown by Spohn in \cite{spohn2000semiclassical} and in \cite{gerard1997homogenization} that in the semiclassical limit $\hbar \rightarrow 0$ the Wigner matrix measure $F$ for the electron component obeys
\begin{equation*}
    \partial_t F  +\{\lambda_-,F\}+i[H_s^{(-)},F] = 0
\end{equation*}
where
\begin{equation}
    H_s^{(-)} := -i[\Pi_-,\{\lambda_-,\Pi_-\}] - \frac{i}{2}\Pi_-\left(\lambda_-\{\Pi_-,\Pi_-\}-\lambda_+\{\Pi_+,\Pi_+\}\right)\Pi_-
    \label{eq:spinor_hamiltonian}
\end{equation}
where $\lambda_{\pm}$ are the two degenerate eigenvalues of the symbol matrix of the Dirac operator and $\Pi_{\pm}$ the respective projections on the eigenspaces. 
This is in fact a special case of a theorem from \cite{gerard1997homogenization} where more general pseudodifferential operators are considered. In \cite{spohn2000semiclassical}, $H_s^{(-)}$ is called \emph{spinor Hamiltonian} and it depends on the positron component as well which can be seen from the last term in \eqref{eq:spinor_hamiltonian}. For the linear Pauli equation the situation is different: There is one degenerate eigenvalue \eqref{eq:eigenvalue_pauli_symbol} whose eigenspace coincides with the whole space and the projection on the eigenspace is the identity $\text{Id}$. In this case, $H_s$ vanishes since
\begin{equation*}
    H_s = -i[\text{Id},\{\lambda,\text{Id}\}] = 0.
\end{equation*}
Therefore the Wigner matrix equation reduces to
\begin{equation*}
    \partial_t F + (\nabla_{\xi} \lambda\cdot\nabla_x F - \nabla_{x} \lambda\cdot \nabla_{\xi} F ) =0
\end{equation*}
Also, $R^{\hbar}_{\text{diag}}$ converges to 
\begin{equation*}
    R_{\text{diag}} = \int_{\mathbb{R}^3_{\xi}} F \dd \xi.
\end{equation*}
The spinor $u$ decouples into two independent Vlasov equations represented by the diagonal elements of the Wigner matrix $F$. This is consistent with the fact that the Stern-Gerlach term \eqref{eq:stern_gerlach_term} which couples the spin components $u_1$ and $u_2$ scales with $\hbar$ and vanishes in the semiclassical limit. 
\end{remark}

\subsection{Notation}


We use notation where $\mathcal{M}(\mathbb{R}^d)$ denotes the positive and bounded measures on $\mathbb{R}^d$. A subscript $w*$ indicates that the space is equipped with the weak star topology. $I=[-T,T]$ denotes a time interval. $W^{k,p}(\mathbb{R}^d)$ are the usual Sobolev spaces and $\mathcal{D}'$ is the space of distributions. For a vector space $X$ we use $X^{n\times n}$ to denote its matrix-valued analogue, i.e. $F\in X^{n\times n}$ means that $F$ is a $n\times n$ matrix with components in $X$. The expression $V_{-}$ denotes the negative part of $V$. $\Tr$ denotes the trace of a $n\times n$ matrix and $\tr$ denotes the trace of a trace class operator on a Hilbert space $\mathcal{H}$. If $M$ is a $n\times n$ matrix with values in $(L^p)^{n\times n}$, then $\|M\|_p$ denotes $\sum_{i,j}\|M_{ij}\|_p$.

\section{Global wellposedness of the Pauli-Poisson equation}
\label{sec:GWP}
In this section we are not interested in the semiclassical limit and therefore choose a scaling where $\hbar=1$. The mixed state Pauli-Poisson equation is given by
\begin{align}
        i\partial_t u_j &= -\frac{1}{2}( \nabla-{i}A)^2u_j + V u_j - \frac{1}{2} (\sigma \cdot B) u_j,\label{eq:PP_Pauli_no_params_mixed}\\
    \Delta V &= -\rho_{\text{diag}} := -\sum_{j=1}^{\infty} \lambda_j |u_j|^2\label{eq:PP_Poisson_no_params_mixed} \\
    u_j(x,0) &= u_{j,I}(x) \in (L^2(\mathbb{R}^3))^2 \label{eq:PP_initial_no_params_mixed},
\end{align}
where $\{u_j\}_{j\in \mathbb{N}}$ is an orthonormal system in $(L^2(\mathbb{R}^3))^2$ and $\{\lambda\}_{j\in \mathbb{N}}$ is a sequence satisfying $\lambda_j \geq 0$ and 
\begin{equation}
   \|\lambda\|_1 :=  \sum_{j=1}^{\infty} \lambda_j =1.
\end{equation}
The global wellposedness in the energy space for the m-SH equation \eqref{eq:magnetic schrödinger hartree} for the pure state case was proven in \cite{michelangeli2015global}. The method is to establish an appropriate energy space where the magnetic Laplacian defines a self-adjoint operator. Then one shows that the Hartree nonlinearity is Lipschitz in the energy space and finally one uses energy conservation to extend the solution globally. In \cite{michelangeli2015global}, the magnetic potential $A$ is assumed to be in $L^2_{\text{loc}}(\mathbb{R}^3)$ which is sufficient for the magnetic Laplacian \eqref{eq:magnetic_laplacian} to be self-adjoint on $L^2(\mathbb{R}^3)$ due to a theorem by Leinfelder and Simader, cf. \cite{leinfelder1981schrodinger}. For the Pauli-Poisson equation we will treat the Stern-Gerlach term as a perturbation of the magnetic Laplacian as in \cite{balinsky2001zero}.

The result in \cite{michelangeli2015global} is established for pure states. Barbaroux and Vougalter proved global wellposedness for the m-SP equation \eqref{eq:MSP_Schrödinger}-\eqref{eq:MSP_data} in \cite{barbaroux2017well}  for mixed states but only for bounded magnetic fields. Global wellposedness in $H^2$ of the SP equation without magnetic field for mixed states was obtained in \cite{brezzi1991three},\cite{illner1994global} and in $L^2$ in \cite{castella1997l2}.  We will prove global wellposedness in the energy space for the Pauli-Poisson equation with rougher magnetic potential than in \cite{barbaroux2017well} for mixed states, thereby extending the results from \cite{barbaroux2017well}, \cite{michelangeli2015global}. We will use the notation
\begin{equation*}
    \mathbf{u} := \{u_j\}_{j\in \mathbb{N}}, \quad
    \mathbf{u}_{I} = \{u_{j,I}\}_{j\in \mathbb{N}}
\end{equation*}
Moreover, we have
\begin{align}
    |\mathbf{u}| = \sum_{j=1}^{\infty} \lambda_j |u_j| && |\mathbf{u}|^2 \equiv \rho_{\text{diag}} = \sum_{j=1}^{\infty} \lambda_j |u_j|^2
\end{align}
and denote the mixed state analogues to $(L^p(\mathbb{R}^3))^2$, $(H^s(\mathbb{R}^3))^2$ and $(H^1_A(\mathbb{R}^3))^2$ (similary to \cite{brezzi1991three},\cite{castella1997l2},\cite{illner1994global}) by
\begin{align*}
\mathbf{L}^p &:= \{\mathbf{u}\colon u_j \in (L^p(\mathbb{R}^3))^2 \, \forall j\in \mathbb{N}; \, \|\mathbf{u}\|^2_{p} := \sum \lambda_j \|u_j\|^2_p < \infty \}, \\
    \mathbf{H}^s &:=\{\mathbf{u}\colon u_j \in (H^s(\mathbb{R}^3))^2 \, \forall j\in \mathbb{N}; \, \|\mathbf{u}\|^2_{H^s} := \sum \lambda_j \|u_j\|^2_{H^s} < \infty \}, \\
    \mathbf{H}^1_A &:=\{\mathbf{u}\colon u_j \in (H^1_A(\mathbb{R}^3))^2 \, \forall j\in \mathbb{N}; \, \|\mathbf{u}\|^2_{H^1_A} := \sum \lambda_j \|u_j\|^2_{H^1_A} < \infty \}. 
\end{align*}
where $H^1_A$ is defined as
\begin{equation}
    H^1_A = \{f\colon \|f\|^2_{H^1_A} := \|(\nabla-iA)f\|_2^2 + \|f\|_2^2 < \infty\}
\end{equation}
for given $A$. The usual inequalities also hold for mixed state spaces, e.g. the Hölder inequality, cf. \cite{castella1997l2}:
\begin{equation*}
    \|\mathbf{u}\mathbf{v}\|_r \leq \|\mathbf{u}\|_p \|\mathbf{v}\|_q,
\end{equation*}
for $1/r = 1/p + 1/q$ and $\mathbf{u} \in \mathbf{L}^p$,$\mathbf{v}\in \mathbf{L}^q$. But also the Sobolev, Young and Hardy-Littlewood-Sobolev inequalities can be easily generalized to mixed state spaces.
When we write $\mathbf{u}\in C^{\infty}_0$ we mean that $u_j\in (C^{\infty}_0)^2$ for all $j\in \mathbb{N}$. Operators act component-wise, e.g.
\begin{equation}
    \nabla \mathbf{u} = \{\nabla u_j\}_{j\in \mathbb{N}}.
\end{equation}
We now adapt the proof from \cite{michelangeli2015global} for the Pauli-Poisson equation. We define the  energy space $\mathcal{H}^1$ for the Pauli equation and show that $|x|^{-1}\ast |\mathbf{u}|^2$ is Lipschitz in $\mathcal{H}^1$. After having obtained a local solution by a fixed point argument we extend the local solution globally using energy conservation. The following lemma is based on \cite{balinsky2001zero}.
\begin{lemma}
\label{thm:self_adjoint_perturbation}
Let $A\in L^2_{\text{\emph{loc}}}(\mathbb{R}^3)$, $|B| \in L^{2}(\mathbb{R}^3)$. Then the sesquilinear form 
\begin{equation}
    h_0(f,g) := \langle H_0 f, g \rangle, \quad f,g \in (C^{\infty}_0(\mathbb{R}^3))^2,
\end{equation}
where $H_0$ is the linear Pauli Hamiltonian defined by \eqref{eq:P_Hamiltonian} is symmetric, closable and non-negative in $(L^2(\mathbb{R}^3))^2$. The associated self-adjoint operator $H_0$ has form domain $\mathfrak{Q}(H_0)$ equal to the completion of $(C^{\infty}_0(\mathbb{R}^3))^2$ with respect to the norm $\|f\|^2_{H^1_A} := \|(\nabla-iA)f\|^2 + \|f\|^2 $.
\begin{proof}
Let $\epsilon >0$ and write $|B| = B_1 + B_2$ with $\|B_1\|_{L^{2}}<\epsilon$ and $\|B_2\|_{L^{\infty}}<C_{\epsilon}$ where $C_{\epsilon}$ is a constant depending on $\epsilon$. Then
\begin{equation*}
    \langle H_0 f,f \rangle = \langle -(\nabla-iA)^2  f,f \rangle + \langle (\sigma \cdot B) f,f \rangle.
\end{equation*}
Moreover,
\begin{align*}
    |\langle (\sigma \cdot B) f,f \rangle| &\leq \langle B_1 f,f \rangle +\langle B_1 f,f \rangle \\
    &\leq \|B_1\|_{2} \|f \|_4^2 + C_{\epsilon} \|f\|_2^2 \\
    &\leq \tilde{C}\epsilon  \|\nabla |f|\|_2^2 + C_{\epsilon} \|f\|_2^2 \\
    &\leq  \tilde{C}\epsilon\|(\nabla-iA) f\|_2^2 + C_{\epsilon} \|f\|_2^2
\end{align*}
The third inequality follows from the Sobolev inequality and the last inequality follows from the diamagnetic inequality. It is known (cf. \cite{leinfelder1981schrodinger}) that for $A\in L^2_{\text{loc}}$, $-(\nabla-iA)^2$ defines a non-negative self-adoint operator with form domain $\mathfrak{Q}(H_0)$. Hence $\sigma \cdot B$ defines a relatively form-bounded perturbation of the form associated to $-(\nabla-iA)^2$ and by the KLMN theorem on relatively form bounded perturbations of closed forms, $h_0$ defines a closable non-negative form with form domain $\mathfrak{Q}(H_0)$.
\end{proof}
\end{lemma}

\begin{definition}[Energy space]
\label{def:energy_space}
Let $A\in L^2_{\text{{loc}}}(\mathbb{R}^3)$, $|B| \in L^{2}(\mathbb{R}^3)$. The energy space for the Pauli equation with magnetic potential $A$ and magnetic field $B$ is given by
\begin{equation*}
    \mathcal{H}^1(\mathbb{R}^3) := \{\mathbf{u}\in \mathbf{L}^2 \colon (\nabla-iA)\mathbf{u} \in \mathbf{L}^2, (\sigma \cdot B)_+^{1/2}\mathbf{u} \in \mathbf{L}^2\} 
\end{equation*}
with associated norm
\begin{equation*}
    \|\mathbf{u}\|^2_{\mathcal{H}^1} :=  \|(\nabla-iA)\mathbf{u}\|^2_2 + \|(\sigma \cdot B)_+^{1/2}\mathbf{u}\|^2_2 + \|\mathbf{u}\|_2^2.
\end{equation*}
\end{definition}

\begin{lemma}[Properties of $\mathcal{H}^1(\mathbb{R}^3)$ and $H_0$] 
Let $A\in L^2_{\text{\emph{loc}}}(\mathbb{R}^3)$, $|B| \in L^{2}(\mathbb{R}^3)$. Then $\mathcal{H}^1(\mathbb{R}^3)$ is a Hilbert space, complete with respect to $\|\cdot\|_{\mathcal{H}^1}$. The norm $\|\cdot\|_{\mathcal{H}^1}$ is equivalent to $\|\cdot\|_{\mathbf{H}^1_A}$. Moreover, $H_0$ is unitary on $\mathcal{H}^1(\mathbb{R}^3)$, i.e. 
\begin{equation*}
    \|e^{-itH_0}\mathbf{u}\|_{\mathcal{H}^1} = \|\mathbf{u}\|_{\mathcal{H}^1},
\end{equation*} for all $f\in \mathcal{H}^1$ and for all $t\in \mathbb{R}$.
\begin{proof}
The norms are equivalent because
\begin{align*}
    \|\mathbf{u}\|_{\mathcal{H}^1}^2 &\leq \|\nabla_A \mathbf{u}\|_2^2 + \|(\sigma\cdot B)^{1/2} \mathbf{u}\|_2^2 + \|\mathbf{u}\|^2_2 \\
    &\lesssim (1+a)\|\nabla_A \mathbf{u}\|_2^2 + (1+b) \|\mathbf{u}\|_2^2
\end{align*}
where $a\in(0,1)$ and $b\geq 0$ are the constants from Lemma \ref{thm:self_adjoint_perturbation} and
\begin{align*}
    \|\mathbf{u}\|_{H^1_A}^2 &\leq \|\nabla_A \mathbf{u}\|_2^2  + \|\mathbf{u}\|^2_2 \\
    &\leq  \|\nabla_A \mathbf{u}\|_2^2 + \|(\sigma\cdot B)^{1/2} \mathbf{u}\|_2^2 + \|\mathbf{u}\|^2_2.
\end{align*}
Unitarity follows from the fact that $H_0$ is self-adjoint on $\mathcal{H}^1(\mathbb{R}^3)$.
\end{proof}
\end{lemma}

\begin{lemma}[Local Lipschitz property]\label{thm:lipschitz} Let $A\in L^2_{\text{\emph{loc}}}(\mathbb{R}^3)$, $|B| \in L^2(\mathbb{R}^3)$. Moreover, let $\mathcal{H}^1(\mathbb{R}^3)$ be the energy space from Definition \ref{def:energy_space}. Then for any $\mathbf{u},\mathbf{v} \in \mathcal{H}^1(\mathbb{R}^3)$, the nonlinearity $(|x|^{-1}\ast|\mathbf{u}|^2)\mathbf{u}$ is locally Lipschitz, i.e.
\begin{equation}
    \|(\frac{1}{|x|}\ast |\mathbf{u}|^2)\mathbf{u}-(\frac{1}{|x|}\ast |\mathbf{v}|^2)\mathbf{v}\|_{\mathcal{H}^1} \leq C(\|\mathbf{u}\|^2_{\mathcal{H}^1}+\|\mathbf{v}\|^2_{\mathcal{H}^1})\|\mathbf{u}-\mathbf{v}\|_{\mathcal{H}^1},
    \label{eq:lipschitz}
\end{equation}
where $C$ is independent of $\mathbf{u}$ and $\mathbf{v}$.
\begin{proof}
Assume that $\mathbf{u},\mathbf{v}\in C^{\infty}_0$ first. Rewrite the magnetic derivative part of the $\mathcal{H}^1$ norm as follows:
\begin{align}
\begin{split}
    \|\nabla_A (\frac{1}{|x|}\ast |\mathbf{u}|^2)\mathbf{u}-(\frac{1}{|x|}\ast |\mathbf{v}|^2)\mathbf{v}\|_2 &\leq \| (\frac{1}{|x|}\ast \nabla (|\mathbf{u}|^2-|\mathbf{v}|^2))\mathbf{u}\|_2 \\ 
    &+ \| (\frac{1}{|x|}\ast \nabla |\mathbf{v}|^2)(\mathbf{u}-\mathbf{v})\|_2 \\
    &+ \|(\frac{1}{|x|}\ast (|\mathbf{u}|^2-|\mathbf{v}|^2))\nabla_A\mathbf{u}\|_2 \\ 
    &+ \| (\frac{1}{|x|}\ast |\mathbf{v}|^2)\nabla_A(\mathbf{u}-\mathbf{v})\|_2,
    \label{eq:lipschitz_split}
\end{split}
\end{align}
since $\nabla_A (fg) = f\nabla_A g + (\nabla f)g$.    Notice that
\begin{equation}
\label{eq:lipschitz_aux_1}
    \left\| (\frac{1}{|x|}\ast (fh))g \right\|_2 \lesssim \|fh\|_{p}^2 \|g\|_{q},
\end{equation}
where $2/p+1/q = 7/6$ due to the Hölder and Hardy-Littlewood-Sobolev inequalities. Choosing $p=2 $ and $q=6$ and using the Sobolev and diamagnetic  inequalities yields
\begin{equation}
\label{eq:lipschitz_aux_2}
    \left\| (\frac{1}{|x|}\ast (fh))g \right\|_2 \lesssim \|f\|_2\|h\|_2 \|\nabla_Ag\|_2.
\end{equation}
It follows that:
\begin{equation}
    \left\| (\frac{1}{|x|}\ast \nabla|f|^2)g \right\|_2 \lesssim \||f|\|_2 \|\nabla |f|\|_2 \|g\|_6 \lesssim \|f\|_2 \|\nabla_Af\|_2\|\nabla_A g\|_2
\end{equation}
We now apply theses inequality to the terms in \eqref{eq:lipschitz_split}. The first term yields
\begin{align}
\begin{split}
    \| (\frac{1}{|x|}\ast \nabla (|\mathbf{u}|^2-|\mathbf{v}|^2))\mathbf{u}\|_2  &\lesssim \|\mathbf{u}\|_2 \|\nabla_A(\mathbf{u}-\mathbf{v})\|_2 \|\nabla_A \mathbf{u}\|_2 \\
    &+\|\mathbf{u}-\mathbf{v}\|_2\|\nabla_A \mathbf{v}\|_2 \|\nabla_A \mathbf{u}\|_2.
    \end{split}
    \label{eq:lipschitz_part1}
\end{align}
For the second term we obtain similarly,
\begin{align}
    \| (\frac{1}{|x|}\ast \nabla |\mathbf{v}|^2)(\mathbf{u}-\mathbf{v})\|_2 \leq \| \mathbf{v}\|_2 \|\nabla_A \mathbf{v}\|_2 \|\nabla_A(\mathbf{u}-\mathbf{v})\|_2.
    \label{eq:lipschitz_part2}
\end{align}
For the third term we have using Young's inequality and the fact that $|x|^{-1} \in L^{3/2}+L^{\infty}$
\begin{align}
\begin{split}
    \|(\frac{1}{|x|}\ast (|\mathbf{u}|^2-|\mathbf{v}|^2))\nabla_A\mathbf{u}\|_2 \lesssim&  \||\mathbf{u}|^2-|\mathbf{v}|^2\|_3 \|\nabla_A \mathbf{u}\|_2 + \||\mathbf{u}|^2-|\mathbf{v}|^2\|_1 \|\nabla_A\mathbf{u}\|_2 \\
    \leq& \|\mathbf{u}-\mathbf{v}\|_6(\|\mathbf{u}\|_6 + \|\mathbf{v}\|_6)\|\nabla_A\mathbf{u}\|_2 \\
    &+ \|\mathbf{u}-\mathbf{v}\|_2(\|\mathbf{u}\|_2 + \|\mathbf{v}\|_2)\|\nabla_A\mathbf{u}\|_2 \\
    \leq& \|\nabla_A(\mathbf{u}-\mathbf{v})\|_2(\|\nabla_A\mathbf{u}\|_2 + \|\nabla_A\mathbf{v}\|_2)\|\nabla_A\mathbf{u}\|_2 \\
    &+ \|\mathbf{u}-\mathbf{v}\|_2(\|\mathbf{u}\|_2 + \|\mathbf{v}\|_2)\|\nabla_A\mathbf{u}\|_2.
    \end{split}
    \label{eq:lipschitz_part3}
\end{align}
For the fourth term we argue similarly and obtain
\begin{align}
\begin{split}
    \| (\frac{1}{|x|}\ast |\mathbf{v}|^2)\nabla_A(\mathbf{u}-\mathbf{v})\|_2 \lesssim \|\nabla_A \mathbf{v}\|_2\|\nabla_A(\mathbf{u}-\mathbf{v})\|_2 + \| \mathbf{v}\|_2\|\nabla_A(\mathbf{u}-\mathbf{v})\|_2.
    \end{split}
    \label{eq:lipschitz_part4}
\end{align}
Finally, we have the spin term which is estimated in the same way as the third term above.
\begin{align}
\begin{split}
    \|(\frac{1}{|x|}\ast (|\mathbf{u}|^2-|\mathbf{v}|^2))(\sigma \cdot B)\mathbf{u}\|_2 \lesssim& \|\nabla_A(\mathbf{u}-\mathbf{v})\|_2(\|\nabla_A\mathbf{u}\|_2 + \|\nabla_A\mathbf{v}\|_2)\|(\sigma \cdot B)\mathbf{u}\|_2 \\
    &+ \|\mathbf{u}-\mathbf{v}\|_2(\|\mathbf{u}\|_2 + \|\mathbf{v}\|_2)\|(\sigma \cdot B)\mathbf{u}\|_2.
    \end{split}
    \label{eq:lipschitz_part5}
\end{align}
From \eqref{eq:lipschitz_part1}-\eqref{eq:lipschitz_part5} we deduce \eqref{eq:lipschitz} for $\mathbf{u},\mathbf{v}\in C^{\infty}_0$. It remains to lift the restriction. Consider two sequences $\{\mathbf{u}_n\},\{\mathbf{v}_n\} \subset C^{\infty}_0$ converging to $\mathbf{u},\mathbf{v}$ in $\mathcal{H}^1$ respectively. It is then easy to see that due to the Lipschitz property of the nonlinearity, $(|x|^{-1}\ast|\mathbf{u}_n|^2)\mathbf{u}_n$ and $(|x|^{-1}\ast|\mathbf{v}_n|^2)\mathbf{v}_n$ are Cauchy sequences in $\mathcal{H}^1$ with respective limits $\mathbf{U}$ and $\mathbf{V}$. Since $\|f\|_2 \leq C \|f\|_{\mathcal{H}^1}$ the convergence is also in $L^2$ and we can deduce that the convergence to $\mathbf{U}$ and $\mathbf{V}$ is pointwise a.e. Then one concludes that $\mathbf{U} = (|x|^{-1}\ast|\mathbf{u}|^2)\mathbf{u}$ and similarly for $\mathbf{V}$. Taking limits allows to conclude that the Lipschitz property holds for all $\mathbf{u},\mathbf{v}\in \mathcal{H}^1$.
\end{proof}
\end{lemma}
Define the Banach space
\begin{equation}
    X_T = \{ \mathbf{u} \in C(I,\mathcal{H}^1(\mathbb{R}^3)) \colon \|\mathbf{u}\|_{L^{\infty}(I,\mathcal{H}^1(\mathbb{R}^3))} \leq M; \, \mathbf{u}(0,x) = \mathbf{u}_I(x) \}
\end{equation}
with the norm
\begin{equation}
    \|\mathbf{u}\|_{X_T} := \|\mathbf{u}\|_{L^{\infty}(I,\mathcal{H}^1(\mathbb{R}^3))} = \sup_{|t|\leq T} (\|\mathbf{u}\|_2 + \|H_0 \mathbf{u}\|_2).
\end{equation}
For every $\mathbf{u}\in X_T$ and every $t\in I$ we introduce the map $\mathbf{u}\mapsto \Phi(\mathbf{u})$ defined by
\begin{equation}
    \Phi(\mathbf{u}) := e^{itH_0}\mathbf{u}_0 - i \int_0^t e^{i(t-s)H_0} (\frac{1}{|x|}\ast |\mathbf{u}(s)|^2)\mathbf{u}(s) \dd s.
\end{equation}
We will show that $\Phi$ is a contraction in $X_T$ and conclude by a fixed point argument.
\begin{lemma}[Local wellposedness] Let $A\in L^2_{\text{\emph{loc}}}(\mathbb{R}^3)$, $|B| \in L^{2}(\mathbb{R}^3)$. Let $\mathbf{u}_I \in \mathcal{H}^1(\mathbb{R}^3)$. For every $M>0$ there is a time $T(M)>0$ such that for all $\mathbf{u}_I$ with $\|\mathbf{u}_I\|_{\mathcal{H}^1} \leq M$ there is a unique $\mathbf{u} \in C([-T(M),T(M)],\mathcal{H}^1(\mathbb{R}^3))$ solving \eqref{eq:PP_Pauli_no_params_mixed}-\eqref{eq:PP_initial_no_params_mixed}. In fact, there exist $T_{\text{\emph{min}}}, T_{\text{\emph{max}}}>0$ such that there is a unique solution $\mathbf{u}\in C([-T_{\text{\emph{min}}},T_{\text{\emph{max}}}],\mathcal{H}^1(\mathbb{R}^3))\cap C^1([-T_{\text{\emph{min}}},T_{\text{\emph{max}}}],\mathcal{H}^{1}(\mathbb{R}^3)^*)$. Moreover, the blow-up alternative holds: If $T_{\text{\emph{max}}}< \infty$ then $\|\mathbf{u}(t)\|_{\mathcal{H}^1}\rightarrow \infty$ as $t\uparrow T_{\text{\emph{max}}}$ (resp. as $t\downarrow -T_{\text{\emph{min}}}$ if $T_{\text{\emph{min}}}<\infty$).
\begin{proof}
Let us show that $\Phi$ maps $X_T$ into $X_T$ for the right choices of $T=T(\bar{M})$ and $\bar{M}$. Observe that for $t \in [0,T(\bar{M}))$
\begin{align*}
    \|\Phi(\mathbf{u})\|_{\mathcal{H}^1} &\leq \|e^{itH_0}\mathbf{u}_I\|_{\mathcal{H}^1} + \int_0^t \|e^{i(t-s)H_0} (\frac{1}{|x|}\ast |\mathbf{u}(s)|^2)\mathbf{u}(s)\|_{\mathcal{H}^1} \dd s \\
    &= \|\mathbf{u}_I\|_{\mathcal{H}^1} + \int_0^t \|(\frac{1}{|x|}\ast |\mathbf{u}(s)|^2)\mathbf{u}(s)\|_{\mathcal{H}^1} \dd s \\
    &\lesssim \|\mathbf{u}_I\|_{\mathcal{H}^1} + \|\mathbf{u}\|_{X_T}^3 T(\bar{M}) \\
    &\leq \bar{M} + (2\bar{M})^3 T(\bar{M}).
\end{align*}
where we have used unitarity of $e^{itH_0}$ on $\mathcal{H}^1$ and the local Lipschitz property. Finally, choosing $T(\bar{M}) \simeq 1/(16\bar{M}^2)$ yields that
\begin{equation*}
    \|\Phi(\mathbf{u})\|_{\mathcal{H}^1} \lesssim 2\bar{M},
\end{equation*}
and thus $\Phi$ maps $X_T$ into $X_T$ for $M=2\bar{M}$. Now we show that $\Phi(\mathbf{u})$ is a contraction on $X_T$. Given $\mathbf{u},\mathbf{v} \in X_T$ and $t\in [0,T(M))$, 
\begin{align*}
    \|\Phi(\mathbf{u})-\Phi(\mathbf{v})\|_{\mathcal{H}^1} &\lesssim \int_0^t \|(\frac{1}{|x|}\ast |\mathbf{u}(s)|^2)\mathbf{u}(s)-(\frac{1}{|x|}\ast |\mathbf{v}(s)|^2)\mathbf{v}(s)\|_{\mathcal{H}^1} \dd s \\
    &\lesssim \int_0^t (\|\mathbf{u}(s)\|^2_{\mathcal{H}^1}+\|\mathbf{v}(s)\|^2_{\mathcal{H}^1})\|\mathbf{u}(s)-\mathbf{v}(s)\|_{\mathcal{H}^1} \dd s \\
    &\leq T(\bar{M}) (\|\mathbf{u}\|_{X_T}^2+\|\mathbf{v}\|_{X_T}^2)\|\mathbf{u}-\mathbf{v}\|_{X_T} \\
    &\leq 8 \bar{M}^2 T(\bar{M})\|\mathbf{u}-\mathbf{v}\|_{X_T} \\
    &= \frac{1}{2}\|\mathbf{u}-\mathbf{v}\|_{X_T}.
\end{align*}
A symmetric argument for $t\in(-T(M),0]$ yields the claim. Finally, we have to check the maximality of the time interval. Define
\begin{align*}
    &T_{\text{max}} := \sup\{T>0 \colon \exists \text{ solution } \mathbf{u} \in C([0,T_{\text{max}}),\mathcal{H}^1(\mathbb{R}^3)) \\
    &T_{\text{min}} := \sup\{T>0 \colon \exists \text{ solution } \mathbf{u} \in C((-T_{\text{min}},0],\mathcal{H}^1(\mathbb{R}^3)).
\end{align*}
By definition, for any interval $J =[-T',T] \subseteq (-T_{\text{min}},T_{\text{max}})$ there is a unique solution $\mathbf{u} \in C(J,\mathcal{H}^1)\cap C^1(J,\mathcal{H}^{1*})$. The uniqueness follows since for any two $u,v \in C(J,\mathcal{H}^1)$ we have by the Duhamel formula, the unitarity of $e^{itH_0}$ and the local Lipschitz property of the nonlinearity,
\begin{align*}
    \|\mathbf{u}(t)-\mathbf{v}(t)\|_{\mathcal{H}^1} &\leq \int_0^t \|(\frac{1}{|x|}\ast |\mathbf{u}(s)|^2)\mathbf{u}(s)-(\frac{1}{|x|}\ast |\mathbf{v}(s)|^2)\mathbf{v}(s)\|_{\mathcal{H}^1} \dd s \\
    &\lesssim (\|\mathbf{u}\|^2_{L^{\infty}(I,\mathcal{H})}+\|\mathbf{v}\|^2_{L^{\infty}(I,\mathcal{H})})\int_0^t \|\mathbf{u}(s)-\mathbf{v}(s)\|_{\mathcal{H}^1} \dd s.
\end{align*}
Applying Gronwall's inequality we can deduce that $\mathbf{u}(t)=\mathbf{v}(t)$ on $[0,T]$ and by an analoguous argument on $[-T',0]$ as well. Finally we present the argument for the blow-up alternative for $T_{\text{max}}$ (the case for $T_{\text{min}}$ follows analoguously). If $T_{\text{max}}<\infty$ assume there is a $M>0$ and a sequence $\{t_n\}\subset \mathbb{R}$ such that $t_n \uparrow T_{\text{max}}$ but $\|\mathbf{u}\|_{\mathcal{H}^1}\leq M$ for all $n$. We can deduce from the wellposedness theory that there is a solution whose initial datum has $\mathcal{H}^1$-norm below $M$ with local existence time $T(M)$. Then we can construct a solution $\mathbf{v} \in C((-T(M),T(M)),\mathcal{H}^1)$ with initial datum $\mathbf{u}(t_k)$ with $t_k$ such that $t_k+T(M) > T_{\text{max}}$. Setting $\tilde{\mathbf{v}}(t) = \mathbf{v}(t-t_k)$ we see that $\tilde{\mathbf{v}}\in C ((t_k-T(M),t_k+T(M)),\mathcal{H}^1)$ is a solution. Since it coincides with $\mathbf{u}$ at $t=t_k$ we can construct a solution with initial datum $\mathbf{u}_I$ in the space $C((-T_{\text{min}},t_k+T(M)),\mathcal{H}^1)$ which is a contradiction to the maximality of $T_{\text{max}}$. The claim follows.
\end{proof}
\end{lemma}

\begin{lemma}[Continuous dependence on the initial data]
Let $A\in L^2_{\text{\emph{loc}}}(\mathbb{R}^3)$, $|B| \in L^2(\mathbb{R}^3)$. Let $\mathbf{u}_{n,I},\mathbf{u}_I \in \mathcal{H}^1(\mathbb{R}^3)$ be initial data satisfying $ \mathbf{u}_{n,I}\rightarrow\mathbf{u}_I$ in $\mathcal{H}^1(\mathbb{R}^3)$ with corresponding unique maximal solutions $\mathbf{u}_n\in C(J_n,\mathcal{H}^1)\cap C^1(J_n,\mathcal{H}^{1*})$ and  $\mathbf{u}\in C(J,\mathcal{H}^1)\cap C^1(J,\mathcal{H}^{1*})$. If $J'\subset J$ is a closed interval then $J'\subset J_n$ for $n$ large enough and $ \mathbf{u}_{n}\rightarrow\mathbf{u}$ in $C(J',\mathcal{H}^1)$.
\begin{proof}
W.l.o.g. assume that $0\in J'$. Set $M:= 2\sup_{t\in J'}\{\|\mathbf{u}(t)\|_{\mathcal{H}^1}\}$. According to the previous lemma, $\mathbf{u} \in C((-T(M),T(M)),\mathcal{H}^1)$ is a solution with local existence time $T(M)\simeq 1/16M^2$ and initial datum satisfying $\|\mathbf{u}_I\|_{\mathcal{H}^1} \leq M$ (since $J'\subset J$ by assumption). By assumption, $\|\mathbf{u}_{n,I}\|_{\mathcal{H}^1} \leq M$ for $n$ large enough, thus $\mathbf{u}_{n}$ is a solution at least on $[-T(M),T(M)]$, whence $[-T(M),T(M)]\subset J \cap J_n$ for $n$ large enough. Evoking Duhamel's formula we see that for $t\in [0,T(M)]$ (the other case follows analoguously),
\begin{equation*}
    \|\mathbf{u}_n(t)-\mathbf{u}(t)\|_{\mathcal{H}^1} \lesssim \|\mathbf{u}_{n,I}-\mathbf{u}_I\|_{\mathcal{H}^1} +\frac{1}{2}\|\mathbf{u}_n-\mathbf{u}\|_{L^{\infty}((-T(M),T(M)),\mathcal{H}^1)}
\end{equation*}
This implies that $\mathbf{u}_n\rightarrow \mathbf{u}$ in $C((-T(M),T(M)),\mathcal{H}^1)$. Since $T(M)$ depends only on $M$ an iteration of this argument covers the whole interval $J'$. We conclude that $J'\subset J_n$ for $n$ large enough and $\mathbf{u}_n\rightarrow \mathbf{u}$ in $C(J',\mathcal{H}^1)$.
\end{proof}
\end{lemma}

\begin{lemma}[Energy and charge conservation]
\label{thm:charge_energy_conservation}
Let $A\in L^2_{\text{\emph{loc}}}(\mathbb{R}^3)$, $|B| \in L^2(\mathbb{R}^3)$ and let $\mathbf{u}\in C((-T_{\text{\emph{min}}},T_{\text{\emph{max}}}),\mathcal{H}^1)\cap C^1((-T_{\text{\emph{min}}},T_{\text{\emph{max}}}),\mathcal{H}^{1*}) $ be the unique maximal solution to the IVP \eqref{eq:PP_Pauli_no_params_mixed}-\eqref{eq:PP_initial_no_params_mixed}. Then the charge
\begin{equation}
    Q(\mathbf{u}) := \|\mathbf{u}\|_2^2 =  \sum_{j=1}^{\infty} \lambda_j \|u_j\|_2^2 = \int \rho_{\text{\emph{diag}}} \dd x
\end{equation}
and the energy
\begin{equation}
    E(\mathbf{u}) := \frac{1}{2}\int_{\mathbb{R}^3} \sum_{j=1}^{\infty} {\lambda_j} |(\sigma \cdot(\nabla-iA))u_j|^2  \dd x  + \frac{1}{2}\int_{\mathbb{R}^3} |\nabla V|^2 \dd x
\end{equation}
are conserved, i.e. $Q(\mathbf{u}(t)) = Q(\mathbf{u}_I)$ and $E(\mathbf{u}(t)) = E(\mathbf{u}_I)$ for all $t \in (-T_{\text{\emph{min}}},T_{\text{\emph{max}}})$.
\begin{proof}
By the continuous dependence on the initial data and since $(C^{\infty}_0(\mathbb{R}^3))^2$ is dense in $\mathcal{H}^1(\mathbb{R}^3)$ we can assume without loss of generality that $\mathbf{u}$ is smooth. Charge conservation follows from multiplying the mixed state Pauli equation by $\overline{u}_j$, multiplying by $\lambda_j$, integrating and taking the imaginary part.
For the energy observe we multiply by $\partial_t \overline{u_j}$, multiply by $\lambda_j$, integrate and take the real part to obtain
\begin{align*}
    \Re \frac{1}{2}\sum_j \lambda_j \int (\partial_t \overline{u_j}) (\sigma\cdot(\nabla-iA))^2 u_j \dd x + \Re \sum_j \lambda_j \int (\partial_t \overline{u_j}) V u_j = 0 
\end{align*}
This is equivalent to
\begin{equation*}
    \partial_t \frac{1}{2}\sum_j \lambda_j \int (\partial_t \overline{u_j}) (\sigma\cdot(\nabla-iA))^2 u_j \dd x + \partial_t \frac{1}{2}\Re \sum_j \lambda_j \int |\nabla V|^2= 0. 
\end{equation*}
Indeed,
\begin{align*}
    \Re \int (\partial_t \overline{u}_j) V u_j \dd x &= \frac{1}{2}\partial_t \int V|u_j|^2 \dd x - \int (\partial_t V) |u_j|^2 \dd x \\
    &= \frac{1}{2}\partial_t \int |\nabla V|^2 \dd x - \Re \iint \frac{1}{|x-y|} (\partial_t \overline{u_j}(y))u_j(y) |u_j(x)|^2 \dd y \dd x \\
    &= \frac{1}{2}\partial_t \int |\nabla V|^2 \dd x - \Re \int (\partial_t \overline{u}_j) V u_j \dd x.
\end{align*}
Thus, $\partial_t E(\mathbf{u}(t)) = 0$.
\end{proof}
\end{lemma}


\begin{proof}[Proof of Theorem \ref{thm:PP_gwp}]
By Lemma \ref{thm:charge_energy_conservation}, $\sup_{t\in \mathbb{R}} \|\mathbf{u}\|_{\mathcal{H}^1} < \infty$. Hence $T_{\text{min}} = T_{\text{max}} = \infty$ by the blow-up alternative.
\end{proof}

\section{Semiclassical limit}
\label{sec:semiclassical_limit}

In this section we prove Theorem \ref{thm:main}. The case of the linear Pauli equation with external electromagnetic potentials $V$ and $A$ is covered in Section \ref{sec:linear pot}. The case of the Pauli-Poisson equation with selfconsistent electric and external magnetic field is treated in Section \ref{sec:nonlinear pot}. The convergence of the current is proved in Section \ref{sec:current}.

Let us derive the Pauli-Wigner equation \eqref{eq:pauli_wigner}. From the von Neumann-Liouville equation \eqref{eq:von_Neumann_Liouville} we obtain that the matrix valued density matrix $R^{\hbar}$ obeys the following equation with $r=x+(\hbar y)/2$ and $s=x-(\hbar y)/2$:
\begin{equation}
    \partial_t R^{\hbar}(r,s,t) = \frac{1}{i\hbar}(H_{r}-{H^{\dagger}_{s}})R^{\hbar}(r,s,t),
\end{equation}
where the subscript indicates that the Hamiltonian is taken with respect to the respective variable and $\dagger$ indicates the conjugate transpose. This means (note that $(\sigma \cdot B)^{\dagger} =(\sigma \cdot B)$),
\begin{align}
\begin{split}
    \partial_t R^{\hbar}(r,s,t) = &\frac{i}{2} [ \hbar (\Delta_r-\Delta_s) + 2(A(r)\nabla_r +A(s)\nabla_s) -  \frac{i}{2\hbar}(A(r)^2-A(s)^2)\\ &-\frac{i}{2}(\sigma\cdot B(r) - \sigma \cdot B(s))  
    -\frac{i}{\hbar}(V^{\hbar}(r,t)-V^{\hbar}(s,t))]R^{\hbar}(r,s,t).
\end{split}
\end{align}
Now we set $T^{\hbar}(x,y,t) := R^{\hbar}(r,s,t)$. Then, using 
\begin{equation*}
    \text{div}_y(\nabla_x T^{\hbar})(x,y,t) = \frac{1}{2} (\Delta_r-\Delta_s)R^{\hbar}(r,s,t)
\end{equation*}
and 
\begin{align*}
    \nabla_r = \frac{1}{2}\nabla_x + \frac{1}{\hbar}\nabla_y, &&
    \nabla_s = \frac{1}{2}\nabla_x - \frac{1}{\hbar}\nabla_y,
\end{align*}
we obtain
\begin{align}
    \begin{split}
        \partial_t T^{\hbar} &= i \hbar \text{div}_y(\nabla_x T^{\hbar}) +\beta[A]\cdot\nabla_xT^{\hbar} -i \delta[A]\cdot \nabla_y T^{\hbar} \\
        &- \frac{i}{2}\delta[A^2]T^{\hbar} - \frac{i\hbar}{2}\delta[\sigma\cdot B]T^{\hbar}-i\delta[V^{\hbar}]T^{\hbar}. 
    \end{split}
\end{align}
where $\beta$ and $\delta$ are defined in \eqref{eq:beta} and \eqref{eq:delta}.
Taking the Fourier transform in $y$ for the first term on the RHS, using partial integration and setting $F^{\hbar} = \mathcal{F}[u^{\hbar}]$ yields:
\begin{equation*}
    i \hbar \mathcal{F}_y (\text{div}_y(\nabla_x T^{\hbar}))(x,\xi,t) = -\xi \cdot \nabla_x F^{\hbar}(x,\xi,t) .
\end{equation*}
Moreover,
\begin{equation*}
    \mathcal{F}_y[\beta[A]\cdot \nabla_x T^{\hbar}] = \mathcal{F}_y[\beta[A]]\ast_\xi \nabla_x F^{\hbar},
\end{equation*}
and 
\begin{equation*}
    -i\mathcal{F}_y[\delta[A]\cdot\nabla_y T^{\hbar}] = \mathcal{F}_y[\delta[A]]\ast_\xi (\xi F^{\hbar})
\end{equation*}
Moreover,
\begin{equation*}
    \mathcal{F}_y[\delta[\cdot]T^{\hbar}] = - (\theta[\cdot]F^{\hbar}),
\end{equation*}
where
\begin{equation*}
    (\theta[\cdot]F^{\hbar})(x,\xi,t) := \frac{1}{(2\pi)^3}\int_{\mathbb{R}^6} \delta[\cdot](x,y,t)F^{\hbar}(x,\eta,t) e^{i -(\xi-\eta)\cdot y} \dd \eta \dd y. 
\end{equation*}
We will always assume that $F^{\hbar}_I$ converges (after extraction of a subsequence) weakly* in $(\mathcal{S}'(\mathbb{R}^6))^2$ to a positive (or null) and bounded measure $f_I$. 

\subsection{Linear case: External electric and magnetic fields}
\label{sec:linear pot}


\begin{proof}[Proof of Theorem \ref{thm:main} \ref{prop:C1}]
We identify the limits. Let us start with the spin term. We need to show for all test functions $\varphi \in \mathcal{S}$
that $\langle \hbar \theta[\sigma\cdot B]F^{\hbar},\varphi\rangle$ is bounded independently of $t\in \mathbb{R}$ and that it converges to zero when $\hbar \rightarrow 0$. Indeed we have
\begin{align*}
    \langle \hbar \theta[\sigma\cdot B]F^{\hbar},\varphi\rangle =& \frac{i}{2} \iint_{(\mathbb{R}^3)^2} F^{\hbar}(x,\eta)\int_{\mathbb{R}^3} \mathcal{F}_{\xi}[\varphi](x,y) \\
    &\cdot e^{i\eta\cdot y} \sigma\cdot (B(x+\frac{\hbar y}{2})- B(x-\frac{\hbar y}{2})) \dd y \dd x \dd \eta \\
    =& \frac{i}{2}\langle F^{\hbar},\psi_1^{\hbar}\rangle_{\mathcal{S}\times \mathcal{S}'}.
\end{align*}
Here,
\begin{equation*}
    \psi^{\hbar}_1(x,\eta) := \int_{\mathbb{R}^3} \mathcal{F}_{\xi}[\varphi](x,y) 
     e^{i\eta\cdot y} \sigma\cdot (B(x+\frac{\hbar y}{2})- B(x-\frac{\hbar y}{2})) \dd y,
\end{equation*}
which is in $\mathcal{S}$ by the choice of $\varphi$ and the assumption that $B\in C(\mathbb{R}^3)$. It is now clear that $\mathcal{F}_{\eta}[\psi^{\hbar}_1]$ converges to zero in $L^1(\mathbb{R}^3_z, C(\mathbb{R}^3_x)) $ if $\hbar \rightarrow 0$. 

For the $\theta[A^2]F^{\hbar}$ term we use the same reasoning with $\psi^{\hbar}_1(x,\eta)$ replaced by
\begin{equation*}
    \psi^{\hbar}_2(x,\eta) := \int_{\mathbb{R}^3} \mathcal{F}_{\xi}[\varphi](x,y) 
     e^{i\eta\cdot y} \frac{1}{\hbar} (A(x+\frac{\hbar y}{2})^2- A(x-\frac{\hbar y}{2})^2) \dd y.
\end{equation*}
Since $A\in C^1(\mathbb{R}^3)$, also $A^2\in C^1(\mathbb{R}^3)$ and we can take the limit in $\hbar$. However we have to be careful since we take the directional derivative with respect to a vector valued function. In fact,
\begin{equation*}
    \lim_{\hbar\rightarrow 0} \frac{1}{\hbar}A(x+\frac{\hbar y}{2})^2- A(x-\frac{\hbar y}{2})^2 = \sum_{i,k=1}^3 A_i y_k \partial_k A_i.
\end{equation*}
Thus,
\begin{align*}
    &\frac{1}{2} \frac{i}{(2\pi)^3}\int_{\mathbb{R}^3} \mathcal{F}_{\xi}[\varphi] (x,y) A_i y_k \partial_k A_ie^{i\eta_k y_k} \dd y 
    = \frac{1}{2} (A_i \partial_k A_i) \nabla_{\eta} \varphi(x,\eta).  
\end{align*}
where we use Einstein summation convention for clarity. We conclude that $\langle \theta[A^2] F^{\hbar},\varphi\rangle$ converges to
\begin{equation*}
    \int_{(\mathbb{R}^3)^2} (A_i \partial_k A_i) \nabla_{\xi}\varphi(x,\xi)\dd f(t). 
\end{equation*}

Next, we consider the $i\mathcal{F}_{y}[\delta[A]]\ast_{\xi} (\xi F^{\hbar})$ term.
We have
\begin{align*}
    \langle i\mathcal{F}_{y}[\delta[A]]\ast_{\xi} (\xi F^{\hbar}),\varphi \rangle &= \frac{1}{(2\pi)^3}\iint F^{\hbar}(x,\eta) \int (\nabla_y e^{i\eta\cdot y}) \delta[A](x,y)\mathcal{F}_{\xi}[\varphi](x,y) \dd y \dd x \dd \eta \\
    &= \frac{1}{(2\pi)^3} \langle F^{\hbar}, \psi^{\hbar}_3\rangle_{\mathcal{S}'\times \mathcal{S}},
\end{align*}
where
\begin{equation*}
    \psi_3^{\hbar} := \int_{\mathbb{R}^3} \mathcal{F}_{\xi}[\varphi](x,y) (\nabla_y e^{i\eta \cdot y}) \delta[A]  \dd y.
\end{equation*}
Since $A\in C^1(\mathbb{R}^3)$ we can deduce that $\psi^{\hbar}_3$ converges in $\mathcal{S}$ to 
\begin{equation*}
    \int \mathcal{F}_{\xi}[\varphi] (\nabla_ye^{i\eta \cdot y})y\cdot \nabla_x A \dd y.
\end{equation*}
Furthermore, 
\begin{align*}
    \frac{1}{(2\pi)^3} \int(\nabla_ye^{i\eta \cdot y})y\cdot \nabla_x A(\mathcal{F}_{\xi}\varphi) \dd y &= -i\nabla_x A\cdot \nabla_{\eta} \int \frac{1}{(2\pi)^3} \mathcal{F}_{\xi} [\varphi] \nabla_y e^{i \eta \cdot y} \dd y \\
    &= i\nabla_x A \cdot \nabla_{\eta} \mathcal{F}^{-1}_y [ \nabla_y ( \mathcal{F}_{\xi} \varphi)(x,y)] \\
    &=\nabla_x A \xi \cdot \nabla_{\eta} \varphi(x,\eta).
\end{align*}
We conclude that $ \langle i\mathcal{F}_{y}[\delta[A]]\ast_{\xi} (\xi F^{\hbar}),\varphi\rangle$ converges to
\begin{equation*}
    \int_{(\mathbb{R}^3)^2} \nabla_x A\xi \cdot \nabla_{\xi} \varphi(x,\xi) \dd f(t).
\end{equation*}

Consider now $\mathcal{F}_y[\beta[A]]\ast_\xi \nabla_x F^{\hbar}$. We have
\begin{align*}
    \langle \mathcal{F}_y[\beta[A]]\ast_\xi \nabla_x F^{\hbar}, \varphi\rangle &= \frac{i}{(2\pi)^3}\iint \nabla_xF^{\hbar}(x,\eta) \int e^{i\eta \cdot y} \beta[A] \mathcal{F}_{\xi}[\varphi](x,y)  \dd y \dd x \dd \eta \\
    &= \frac{i}{(2\pi)^3} \langle \nabla_x F^{\hbar}, \psi^{\hbar}_4\rangle 
\end{align*}
where
\begin{equation*}
    \psi_4^{\hbar} := \int_{\mathbb{R}^3} \mathcal{F}_{\xi}[\varphi](x,y) e^{i\eta \cdot y} \beta[A]  \dd y.
\end{equation*}
We easily see similarly to the reasoning before that $\psi_4^{\hbar}$ converges in $\mathcal{S}$ to 
\begin{equation*}
    \int_{\mathbb{R}^3}\mathcal{F}_{\xi}[\varphi](x,y)A(x) e^{i\eta\cdot y} \dd y,
\end{equation*}
which is equivalent to
\begin{equation*}
    A(x)\varphi(x,\eta).
\end{equation*}
Finally we consider the $V$-term we proceed as in the case for $A^2$ and observe that under the assumption $V\in C^1$
\begin{equation*}
    \langle \theta[V] F^{\hbar},\varphi\rangle = i \langle F^{\hbar}, \psi_5^{\hbar}\rangle_{\mathcal{S}\times\mathcal{S}'},
\end{equation*}
where
\begin{equation*}
    \psi^{\hbar}_5(x,\eta) := \int_{\mathbb{R}^3} \mathcal{F}_{\xi}[\varphi](x,y) 
     e^{i\eta\cdot y}\frac{1}{\hbar} (V(x+\frac{\hbar y}{2})- V(x-\frac{\hbar y}{2})) \dd y,
\end{equation*}
converges to 
\begin{equation*}
    \int \nabla_x V \nabla_{\xi} \varphi(x,\xi)\dd f(t).
\end{equation*}
After a change of variables $p=\xi-A(x)$ we conclude $F^{\hbar}$ converges to $F$ in the sense of distributions and $F$ solves the Vlasov equation \eqref{eq:VLF_beforeCOV}. 

For the part about the Hamiltonian flow consider the Hamiltonian
\begin{equation}
    E:= H(x(t),\xi(t)) = H(x(0),\xi(0))  = \frac{1}{2}|\xi(t)+A(x(t))|^2 + V(x(t)).
    \label{eq:hamiltonian definition}
\end{equation}
which depends only on the initial conditions. We claim that the Hamiltonian flow $H_t$ is well-defined. Indeed, the condition $A,V \in C^{1,1}$ implies global existence in the neighborhood of every point for the Cauchy problem \eqref{eq:hamiltonian equations}. Then, by \eqref{eq:V bounded from below} and \eqref{eq:hamiltonian definition},
\begin{equation*}
    E > \frac{1}{2}|\xi + A(x)|^2 - C_1(1+|x|^2)
\end{equation*}
Thus,
\begin{equation}
    |\xi+A(x)|^2 < 2E+ 2C_1 +2C_1|x|^2.
    \label{eq:xi+A bounded}
\end{equation}
Then,
\begin{align*}
    \frac{d}{dt}|x(t)|^2 = 2\dot{x}(t)x(t) &= (\xi(t) + A(x(t))) x(t) \\ &\leq |\xi(t)+A(x(t))|^2 + |x(t)|^2 \\
    &\leq 2E +2C_1 +(2C_1+1)|x(t)|^2
\end{align*}
By Grönwall's inequality we conclude that $|x(t)|^2$ is bounded by an exponential. Returning to \eqref{eq:xi+A bounded}, we use
\begin{equation*}
    |\xi(t)| < |\xi(t) +A(x(t))| + |A(x(t))|
\end{equation*}
and the assumption that $|A(x(t))|$ is exponentially bounded to conclude that there is a $K>0$
\begin{equation}
    |x(t)|+|\xi(t)| \lesssim e^{Kt}
\end{equation}
This implies that the Hamiltonian flow is well defined globally on $\mathbb{R}^6$ and for all $t\in \mathbb{R}$. 

Let now $\psi(x,\xi,t) := \psi_0 \circ H_t(x,\xi)$ with $\psi_0 \in C_0^{\infty}(\mathbb{R}^6)$. Then $\psi$ has compact support in $(x,\xi)$ uniformly in $t$ for $t<\infty$  and is Lipschitz on $(-T,T)\times \mathbb{R}^6$ for $T< \infty$. Moreover,
\begin{equation*}
    \partial_t \psi =  \xi\cdot \nabla_x \psi + A(x) \cdot \nabla_x \psi - (\nabla_xA)\xi \cdot \nabla_{\xi} \psi -(\nabla_xA)A \cdot\nabla_x \psi - \nabla_x V \cdot \nabla_x \psi,
\end{equation*}
which, after a change of variables $p=\xi-A(x)$, is equivalent to
\begin{equation}
    \partial_t \psi =  p\cdot \nabla_x \psi - (\nabla_x V +p\times B)\cdot \nabla_x \psi.
\end{equation}
This implies (omitting the details which are specified in \cite{lions1993mesures}) that for all $\varphi \in C_b(\mathbb{R}^6)$,
\begin{equation}
    \iint_{\mathbb{R}^6} \varphi df(t) = \iint_{\mathbb{R}^6} \varphi \circ H_t df_I.
\end{equation}
which is what we wanted to show.
\end{proof}

\subsection{Nonlinear case: Self-consistent electric field, external magnetic field}
\label{sec:nonlinear pot}

We now turn to the proof of the semiclassical limit in the nonlinear case where $V^{\hbar}$ is given by the Poisson equation \eqref{eq:PP_Poisson_unscaled}. We need the following observations.

\begin{lemma}
\label{thm:energy_bounded}
Let $H_0$ be the linear Pauli Hamiltonian from \eqref{eq:P_Hamiltonian}. If 
\begin{equation*}
    E(0) := \tr(H_0R^{\hbar}) + \frac{1}{2}\iint_{\mathbb{R}_x^3 \times \mathbb{R}_y^3} \frac{\rho^{\hbar}_{\text{\emph{diag}},I}(x) \rho_{\text{\emph{diag}},I}^{\hbar}(y)}{|x-y|} \dd x \dd y,
\end{equation*}
is bounded independently of $\hbar$ then the total energy
\begin{equation*}
    E(t) := \tr(H_0R^{\hbar}) + \frac{1}{2}\iint_{\mathbb{R}_x^3 \times \mathbb{R}_y^3} \frac{\rho^{\hbar}_{\text{\emph{diag}}}(x) \rho_{\text{\emph{diag}}}^{\hbar}(y)}{|x-y|} \dd x \dd y,
\end{equation*}
is bounded independently of $\hbar$.
\begin{proof}
This follows from Lemma \ref{thm:charge_energy_conservation} by rewriting the energy in the density matrix formulation. In fact,
\begin{equation*}
     \tr(H_0R^{\hbar}) = \frac{1}{2} \sum_{j=1}^{\infty} \lambda_j^{\hbar} \|\sigma \cdot(\hbar \nabla -iA)u_j^{\hbar}\|_{2}^2,
\end{equation*}
and
\begin{equation*}
     \frac{1}{2}\iint_{\mathbb{R}_x^3 \times \mathbb{R}_y^3}\frac{\rho^{\hbar}_{\text{{diag}}}(x) \rho_{\text{diag}}^{\hbar}(y)}{|x-y|} \dd x \dd y = \frac{1}{2}\int_{\mathbb{R}^3_x} |\nabla V^{\hbar}|^2 \dd x
\end{equation*}
\end{proof}
\end{lemma}
Now define
\begin{align}
    R(x,y) := \sum_j \lambda_j u_j(x) \otimes \overline{u_j(y)} && \rho_{\text{{diag}}}(x) := \sum \lambda_j |u_j(x)|^2
\end{align} 
where $\lambda = \{\lambda_j\}_{j\in \mathbb{N}}$ is a decreasing sequence with $\lambda_j \geq 0$ for all $j \in \mathbb{N}$ and $\{u_j\}$ is an orthonormal basis of $L^2(\mathbb{R}^3)^2$. Let
\begin{equation}
    \overline{H}_0 := \frac{1}{2}(\sigma\cdot(\nabla-iA))^2
    \label{eq:pauli operator withouth h}
\end{equation}

\begin{lemma}
\label{thm:lieb_thirring}
Let $B\in L^{7/2}(\mathbb{R}^3)$ and let \begin{align*}
    \tr(-\overline{H}_0 R) := \sum_j \lambda^{\hbar}_j\int_{\mathbb{R}^3} |(\sigma\cdot(\nabla-iA)) u_j|^2 \dd x, &&  \|\lambda^{\hbar}\|_{2} = \sum_{j} (\lambda^{\hbar}_j)^2,
\end{align*} 
Then $R_{\text{\emph{diag}}}, \rho^{\hbar}_{\text{\emph{diag}}} \in L^{7/5}(\mathbb{R}^3)$ with 
\begin{align}
    \|R_{\text{\emph{diag}}}\|_{7/5} &\leq \|\lambda^{\hbar}\|_2^{4/7} \tr(-\overline{H}_0 R)^{3/7} \\
    \|\rho_{\text{\emph{diag}}}\|_{7/5} &\leq \|\lambda^{\hbar}\|_2^{4/7} \tr(-\overline{H}_0 R)^{3/7}
    \label{eq:lieb_thirring}
\end{align}
In particular, if $\lambda^{\hbar}$ satisfies Assumption \ref{thm:remark_pure_states} it holds that $R_{\text{\emph{diag}}}, \rho^{\hbar}_{\text{\emph{diag}}} \in L^{7/5}(\mathbb{R}^3)$ independently of $\hbar$.
\begin{proof}
The main ingredient for the proof is the following magnetic Lieb-Thirring inequality from \cite{shen1998moments}. Let $\overline{H}_0$ be the Pauli operator from \eqref{eq:pauli operator withouth h} and denote by $\mu_j$ its negative eigenvalues. Let $\gamma \geq 1$ and $q>3/2$. Then
\begin{equation}
    \sum_j |\mu_j|^{\gamma} \leq C_1(\gamma,q) \int V_{-}^{3/2 + \gamma}\dd x + C_2(\gamma,q) \|B\|_{3q/2}^{3/2} \left(\int V_{-}^{\gamma q'}\dd x\right)^{1/q'}
    \label{eq:magnetic_lieb_thirring}
\end{equation}
We will use an argument from \cite{lions1993mesures} to prove the claim. Define the operator $\overline{H}_0-tR^{\alpha}$ where $t>0$ will be determined later. Then 
\begin{equation*}
    \tr((\overline{H}_0-tR^{\alpha})R) = \sum_{j}\lambda^{\hbar}_j \int |(\sigma \cdot(\nabla-iA))u_j|^2 - tR^{\alpha} |u_j|^2 \dd x \geq \sum_j \lambda^{\hbar}_j |\mu_j|.
\end{equation*}
By the Hölder inequality,
    \begin{equation*}
        t\int R^{\alpha+1} \dd x \leq \tr(\overline{H}_0R) + \|\lambda^{\hbar}\|_{p} \left(\sum |\mu_j|^{p'}\right)^{1/p'}
    \end{equation*}
    Now we use the magnetic Lieb-Thirring inequality \eqref{eq:magnetic_lieb_thirring} and obtain
    \begin{align*}
        t\int R^{\alpha+1} \dd x \leq \tr(\overline{H}_0R) &+ \|\lambda^{\hbar}\|_{p} C_1 \left(\int (tR^{\alpha})^{3/2 + p'}\right)^{1/p'} \\ &+ \|\lambda^{\hbar}\|_{p}C_2 \|B\|_{3q/2}^{3/2} \left(\int (tR^{\alpha})^{p' q'}\right)^{1/(q'p')}
    \end{align*}
    where $C_1$ and $C_2$ depend on $q$ and $p'$. Choosing $p=p'=2$ and setting $\alpha+1 = 3\alpha/2 + 2\alpha$ yields $\alpha = 2/5$. Therefore,
    \begin{align*}
        t\int R^{7/5} \dd x \leq \tr(\overline{H}_0R) &+ tt^{3/4}\|\lambda^{\hbar}\|_{2} C_1 \left(\int R^{7/5}\right)^{1/2} \\ &+ t\|\lambda^{\hbar}\|_{2}C_2 \|B\|_{3q/2}^{3/2} \left(\int R^{ 4q'/5}\right)^{1/(2q')}
    \end{align*}
    Now the power of $R$ in the last integral must be equal to $7/5$ so $q'=7/4$ and $q=7/3>3/2$. By assumption, $B\in L^{7/2}$, so we can choose $3q/2=7/2$. Together with the choice 
    \begin{equation*}
        t^{3/4} = \frac{1}{2C_1\|\lambda^{\hbar}\|_2} \left( \int R^{7/5} \right)^{1/2}
    \end{equation*}
    we obtain
\begin{equation*}
        \frac{1}{2}t\int R^{7/5} \dd x \leq \tr(\overline{H}_0R) + t\|\lambda^{\hbar}\|_{2}C_3 \left(\int R^{7/5}\right)^{2/7}.
    \end{equation*}
    where $C_3$ depends on $\|B\|_{7/2}$. It follows that
    \begin{equation}
        \|R\|_{7/5}^{5/3} \lesssim_{C_1,C_3}\|\lambda^{\hbar}\|_2^{4/3} \tr(\overline{H}_0R) + \|\lambda_2 \|R\|_{7/5}^{4/3}. 
    \end{equation}
    Now since the power of the second term on the RHS is smaller than on the LHS we can absorb it into the LHS and obtain
    \begin{equation*}
            \|R\|_{7/5}^{5/3} \lesssim_{C_1,C_3} \|\lambda^{\hbar}\|_2^{4/3} \tr(\overline{H}_0R).
    \end{equation*}
    For the last statement of the lemma we use \eqref{eq:lieb_thirring} and Assumption \ref{thm:remark_pure_states} to obtain
    \begin{equation*}
        \|R^{\hbar}_{\text{diag}}\|_{7/5} \leq C \hbar^{6/7} \tr(\overline{H}_0R)^{4/7}
    \end{equation*}
    Now
    \begin{align*}
        \tr(\overline{H}_0 R)^{4/7} &\leq \left(\frac{1}{2\hbar^2} \sum_j \lambda_j^{\hbar} \int | (\sigma \cdot \hbar \nabla) u_j|^2 \dd x + \frac{1}{2}\sum_j \int \lambda_j^{\hbar}|(\sigma \cdot A) u_j|^2 \dd x\right)^{4/7} \\
        &\lesssim \frac{1}{\hbar^{6/7}} \left(\sum_j \lambda_j^{\hbar} \int | (\sigma \cdot ( \hbar\nabla-iA)) u_j|^2 \dd x\right)^{4/7} + \left(\sum_j \int \lambda_j^{\hbar}|(\sigma \cdot A) u_j|^2 \dd x\right)^{4/7} \\
        &\leq \frac{1}{\hbar^{6/7}} K^{4/7} + \|A^2\|_{6}\|\rho^{\hbar}_{\text{diag}}\|_{{6/5}} )^{3/7} \\ &\leq \frac{1}{\hbar^{6/7}} K^{4/7}+ C\|\rho^{\hbar}_{\text{diag}}\|_{{7/5}}^{1/4}
    \end{align*} 
    where $C$ does not depend on $\hbar$ and where $K$ denotes the kinetic energy of the Pauli-Poisson equation which is bounded independently of $\hbar$ due to Lemma \ref{thm:energy_bounded}. By combining the last two estimates and observing that the powers of $\hbar$ cancel yields the claim. Obviously all calculations hold for $\rho^{\hbar}_{\text{diag}} = \Tr(R^{\hbar})$ as well.
\end{proof}

\end{lemma}

\begin{proof}[Proof of Theorem \ref{thm:main} \ref{thm:semiclassical_limit_nonlinear}]
By assumption, $\|F_I^{\hbar}\|_{2}$ is bounded independently of $\hbar$ (this is due to \eqref{eq:weight_conditionC}). By conservation of $\|F^{\hbar}\|_2$ we have that $F^{\hbar} \in L^2$ independently of $\hbar$. Thus we can extract a subsequence such that $F^{\hbar}$ converges weakly* to some 
\begin{equation*}
  F \in  C_b(\mathbb{R},\mathcal{M}_{w*}^{2\times 2})\cap   L^{\infty}(\mathbb{R},L^1\cap L^2(\mathbb{R}_x^3\times \mathbb{R}_{\xi}^3)^{2\times2}).
\end{equation*}
Now let $\psi \in \mathcal{S}(\mathbb{R}^3_x \times \mathbb{R}^3_{\xi})$. With the notation of Proposition \ref{prop:C1} we have to show that
\begin{align}
    \psi(x,z) \frac{\hbar}{2} \delta[\sigma \cdot B] &\longrightarrow 0, \\
    \psi(x,z) \frac{1}{2} \delta[A^2] &\longrightarrow   \psi(x,z) z\cdot (\nabla_x A)A(x),\\
    i z \psi(x,z) \delta[A] &\longrightarrow iz \psi(x,z)z\cdot \nabla_x A,\\
    \psi(x,z) \beta[A] &\longrightarrow \psi(x,z)A(x) \\
    \psi(x,z) \delta[V^{\hbar}] &\longrightarrow \psi(x,z)z\cdot \nabla_x V.
\end{align}
in $L^2(\mathbb{R}^6)$ as $\hbar \rightarrow 0$. The first one is immediately clear since $B\in L^2_{\text{loc}}$. For the second one we have to ensure that 
\begin{equation*}
    (\nabla_x A) A\in L^2_{\text{loc}}(\mathbb{R}^3).
\end{equation*}
The assumption on $A$ implies 
\begin{equation*}
    A\in L^{12}_{\text{loc}}(\mathbb{R}^3), \quad \nabla_x A\in L^{12/5}_{\text{loc}}(\mathbb{R}^3),
\end{equation*}
by Sobolev's inequality. The claim follows from Hölder's inequality. The third and fourth convergences are clear since in particular $A\in H^1_{\text{loc}}(\mathbb{R}^3)$.

It remains to show that $V^{\hbar}$ converges in $H^1_{\text{loc}}(\mathbb{R}^3)$ which is equivalent to showing that $\nabla V^{\hbar} = (\nabla |x|^{-1})\ast \rho^{\hbar}_{\text{diag}}$ converges in $L^2_{\text{loc}}(\mathbb{R}^3)$. Note that 
\begin{equation*}
    \nabla |x|^{-1} \in L^{3/2,\infty}(\mathbb{R}^3)
\end{equation*}
so in particular it is in $L^{14/11}(\mathbb{R}^3) + L^q(\mathbb{R}^3)$ for $14/11 < q < \infty$ since $3/2 > 14/11$. By Lemma \ref{thm:lieb_thirring}, $\rho^{\hbar}_{\text{diag}}$ is bounded independently of $\hbar$ in $L^{7/5}(\mathbb{R}^3)$ and by Young's inequality we conclude that $\nabla V^{\hbar} \in L^2_{\text{loc}}(\mathbb{R}^3)$. In order to obtain convergence in $L^{\infty}((0,T),L^2_{\text{loc}}(\mathbb{R}^3_x))$ we use a compactness argument. By the continuity equation we have
\begin{equation*}
    \partial_t \nabla V^{\hbar} = \nabla (-\Delta)^{-1}(\partial_t \rho^{\hbar}_{\text{diag}}) \simeq  J^{\hbar}.
\end{equation*}
A similar argument for $J^{\hbar}$ as for $\rho^{\hbar}_{\text{diag}}$ in Lemma \ref{thm:lieb_thirring} implies that $J^{\hbar}$ is bounded in $L^{7/6}(\mathbb{R}^3)$ independently of $\hbar$. Therefore $\partial_t \nabla V^{\hbar} \in L^{7/6}(\mathbb{R}^3)$ and we conclude by the Aubin-Lions lemma that $V^{\hbar}$ converges strongly in $C((0,T),H^1_{\text{loc}}(\mathbb{R}^3_x))$.
The last part is to show that $\rho^{\hbar}$ is $\hbar$-oscillatory. Observe that $\tr(H\rho^{\hbar}) \leq C$ implies that
\begin{equation*}
    \iint f^{\hbar}(x,\xi)|\xi - A(x)|^2 \dd x \dd \xi \leq C
\end{equation*}
On the other hand,
\begin{align*}
    \iint f^{\hbar} |\xi|^2 \dd x \dd \xi &\leq \iint f^{\hbar} |\xi -A|^2 \dd x \dd \xi + \iint f^{\hbar}A^2 \dd x \dd \xi \\
    &\leq \tr(H_0 R^{\hbar}) + \|f^{\hbar}\|_{p'}\|A^2\|_p
\end{align*}
where $p'$ is dual to $p$. Now take for instance $p=6, p'=6/5$ and observe that $f^{\hbar} \in L^r$ for $r\in [1,2]$ and $A^2\in L^6$. Therefore we can evoke Lemma \ref{thm:suff_cond_h_osc} and the last statement of the Theorem follows from this and the last statement in Theorem \ref{thm:wigner measure}.
\end{proof}

\subsection{Current density}
\label{sec:current}

\begin{proof}[Proof of Theorem \ref{thm:main} \ref{thm_Pauli_current}]
One easily checks (using $\nabla \times (\overline{u}\sigma u) = \tr((\nabla \times \sigma \overline{u})\otimes u) + \tr( \overline{u}\otimes \nabla \times \sigma u) $) that
\begin{equation*}
    {J}^{\hbar} = \int_{\mathbb{R}^3_{\xi}} (\xi-A(x)) {f}^{\hbar}(x,\xi) \dd \xi - \hbar \nabla_x  \times\int_{\mathbb{R}^3_{\xi}} \tr(\sigma F^{\hbar})(x,\xi) \dd \xi,
\end{equation*}
Let $\varphi \in C^{\infty}_0$ and set
\begin{equation*}
    j^{\hbar} = \int_{\mathbb{R}^3_{\xi}} \xi f^{\hbar} \dd \xi, \quad
    j := \int \xi f \dd \xi.
\end{equation*}
Again, $j^{\hbar}$ is bounded in $L^{7/6}(\mathbb{R}^3)$ independently of $\hbar$ and
\begin{equation*}
    \iint j^{\hbar} \varphi \dd x \dd t \rightarrow \iint j \varphi \dd x \dd t,
\end{equation*}
as $\hbar \rightarrow 0$. Since $\rho^{\hbar}_{\text{diag}} \in L^{7/5}(\mathbb{R}^3)$ independently of $\hbar$ and $A\in L^{12}(\mathbb{R}^3)$ (by the Sobolev inequality) we have that
\begin{equation*}
    \iint A\rho_{\text{diag}}^{\hbar} \varphi \dd x \dd t \rightarrow \iint A\rho_{\text{diag}} \varphi \dd x \dd t.
\end{equation*}
since $A\varphi \in L^{7/2}$. By the same reasoning we obtain that the spin term is $o(1)$ as $\hbar \rightarrow 0$. 
\end{proof}

\appendix

\section{Wigner measures}
\label{sec:wigner measures}

Let 
\begin{align*}
    \hat{\rho}^{\hbar}(\xi,\eta) := \mathcal{F}_x[\overline{\mathcal{F}_y}[\rho^{\hbar}(x,y)]] = \sum_{j=1}^{\infty} \lambda_j^{\hbar} \widehat{u_j^{\hbar}}(\xi)\overline{\widehat{u_j^{\hbar}}(\eta)}
\end{align*}
and
\begin{equation*}
    \hat{\rho}^{\hbar}_{\text{diag}}(\xi) := \hat{\rho}^{\hbar}(\xi,\xi) =
        \sum_{j=1}^{\infty} \lambda_j^{\hbar} |\widehat{ u_j^{\hbar}}(\xi)|^2
\end{equation*}
A sequence of density matrices $\rho^{\hbar}$ is said to be \emph{$\hbar$-oscillatory} if for all $\varphi$ continuous with compact support,
    \begin{equation}
        \limsup_{\hbar\rightarrow 0} \int_{|\xi|\geq R/\hbar} \widehat{\varphi \rho^{\hbar}_{\text{diag}}}(\xi)\dd \xi  \rightarrow 0   
        \label{eq:epsilon_oscillatory}
    \end{equation}
    as $R \to \infty$. By $\widehat{\varphi \rho_{\text{diag}}^{\hbar}}(\xi)$ we mean
    \begin{equation*}
        \sum_{j=1}^{\infty} \lambda_j^{\hbar} |\widehat{\varphi u_j^{\hbar}}(\xi)|^2
    \end{equation*}
It is \emph{compact at infinity} if 
    \begin{equation}
    \label{eq:compact at infinity}
        \limsup_{\hbar\rightarrow 0} \int_{|x|\geq R} \rho_{\text{diag}}^{\hbar}(x)\dd x  \rightarrow 0   
    \end{equation}
    as $R \to \infty$.
For reference of these two definitions consider \cite{gerard1997homogenization}. We start with the convergence of the Wigner and Husimi functions $f^{\hbar}$ and $\tilde{f}^{\hbar}$.
\begin{proposition}
\label{thm:wigner measure}
Let $f^{\hbar}$, $F^{\hbar}$ and $\tilde{f}^{\hbar}$, $\tilde{F}^{\hbar}$ be the Wigner and Husimi functions respectively their matrix versions corresponding to a density matrix $\rho^{\hbar}$, respectively matrix valued density matrix $R^{\hbar}$. Denote by $f$ and $\tilde{f}$ their respective limits as $\hbar \rightarrow 0$ (resp. $F$ and $\Tilde{F}$). Then 
\begin{equation*}
f \equiv \tilde{f}, \quad F\equiv \tilde{F}
\end{equation*}
In particular, $f$ (resp. $F$) is a positive Radon measure (resp. matrix valued measure) on $\mathbb{R}^d_x \times \mathbb{R}^d_{\xi}$. Moreover, if ${\rho}^{\hbar}$ (resp. $R^{\hbar}$) is $\hbar$-oscillatory, then
\begin{equation*}
    \rho_{\text{\emph{diag}}}(x) = \int_{\mathbb{R}^3_{\xi}} f(x,\xi) \dd \xi, \quad R_{\text{\emph{diag}}}(x) = \int_{\mathbb{R}^3_{\xi}} F(x,\xi) \dd \xi.
\end{equation*}
\begin{proof}
Theorem III.1 and Theorem III.2 in \cite{lions1993mesures} and Proposition 1.7 in \cite{gerard1997homogenization}.
\end{proof}
\end{proposition}
A sufficient condition for \eqref{eq:epsilon_oscillatory} is provided by the following
\begin{lemma}
\label{thm:suff_cond_h_osc}
Let $s>0$ such that 
\begin{equation}
    \hbar^2 \tr(-\Delta \rho^{\hbar}) \leq C
    \label{eq:suff_cond_tight_comp},
\end{equation}
or equivalently,
\begin{equation*}
    \sum_{j=1}^{\infty} \lambda_j^{\hbar} \hbar^2 \int |\nabla u_j^{\hbar}|^2 \dd x \leq C.
\end{equation*}
Then ${\rho}^{\hbar}$ is $\hbar$-oscillatory.
\begin{proof}
By Plancherel, \eqref{eq:suff_cond_tight_comp} is equivalent to
\begin{equation*}
    \sum_{j=1}^{\infty} \lambda_j^{\hbar} \int \hbar^{2} |\xi|^{2s} |\widehat{u_j^{\hbar}}(\xi)|^2 \dd \xi \leq C
\end{equation*}
It follows that
\begin{equation*}
    \int_{|\xi|>R/\hbar} \hat{\rho}^{\hbar}_{\text{diag}}(\xi) \dd \xi \leq CR^{-2},
\end{equation*}
Letting $R\rightarrow \infty$ proves the claim.
\end{proof}
\end{lemma}

\section*{Acknowledgement}

We acknowledge support from the Austrian Science Fund (FWF) via the grants SFB F65 and W1245 and by the Vienna Science and Technology Fund (WWTF) project MA16-066 "SEQUEX".

We thank P. Gérard (Univ. Paris-Saclay) for many important discussions and his hospitality at Laboratoire de Mathématiques d'Orsay.

\bibliography{semiclassics.bib}

\begin{thebibliography}{10}

\bibitem{alazard2007semi}
T.~Alazard and R.~Carles.
\newblock Semi-classical limit of {S}chr{\"o}dinger--{P}oisson equations in
  space dimension $n\geq 3$.
\newblock {\em J. Diff. Eq.}, 233(1):241--275, 2007.

\bibitem{balinsky2001zero}
A.~Balinsky and W.~D. Evans.
\newblock On the zero modes of {P}auli operators.
\newblock {\em J. Func. Anal.}, 179(1):120--135, 2001.

\bibitem{barbaroux2016existence}
J.-M. Barbaroux and V.~Vougalter.
\newblock Existence and {N}onlinear {S}tability of {S}tationary {S}tates for
  the {M}agnetic {S}chr{\"o}dinger-{P}oisson {S}ystem.
\newblock {\em J. Math. Sci.}, 219(6), 2016.

\bibitem{barbaroux2017well}
J.-M. Barbaroux and V.~Vougalter.
\newblock On the {W}ell-posedness of the {M}agnetic {S}chr{\"o}dinger-{P}oisson
  {S}ystem in $\mathbb{R}^3$.
\newblock {\em Math. Mod. Nat. Phen.}, 12(1):15--22, 2017.

\bibitem{bechouche1998semi}
P.~Bechouche, N.~J. Mauser, and F.~Poupaud.
\newblock ({S}emi)-nonrelativistic limits of the {D}irac equation with external
  time-dependent electromagnetic field.
\newblock {\em Comm. Math. Phys.}, 197(2):405--425, 1998.

\bibitem{brezzi1991three}
F.~Brezzi and P.~A. Markowich.
\newblock The three-dimensional {W}igner-{P}oisson problem: {E}xistence,
  uniqueness and approximation.
\newblock {\em Math. Meth. Appl. Sc.}, 14(1):35--61, 1991.

\bibitem{castella1997l2}
F.~Castella.
\newblock L2 solutions to the {S}chr{\"o}dinger--{P}oisson system: existence,
  uniqueness, time behaviour, and smoothing effects.
\newblock {\em Math. Mod. Meth. Appl. Sc.}, 7(08):1051--1083, 1997.

\bibitem{diperna1989global}
R.~J. DiPerna and P.-L. Lions.
\newblock Global weak solutions of {V}lasov-{M}axwell systems.
\newblock {\em Comm. Pure Appl. Math.}, 42(6):729--757, 1989.

\bibitem{erdHos1997semiclassical}
L.~Erd{\H{o}}s and J.~Solovej.
\newblock Semiclassical eigenvalue estimates for the {P}auli operator with
  strong non-homogeneous magnetic fields: {I}{I}. {L}eading order asymptotic
  estimates.
\newblock {\em Comm. Math. Phys.}, 188:599--656, 1997.

\bibitem{frenod1998homogenization}
E.~Fr{\'e}nod and E.~Sonnendr{\"u}cker.
\newblock Homogenization of the {V}lasov equation and of the
  {V}lasov--{P}oisson system with a strong external magnetic field.
\newblock {\em Asympt. Analysis}, 18(3-4):193--213, 1998.

\bibitem{gerard1991mesures}
P.~G{\'e}rard.
\newblock Mesures semi-classiques et ondes de {B}loch.
\newblock {\em S{\'e}m. {E}q. D{\'e}r. Part.}, 16:1--19, 1991.

\bibitem{gerard1991microlocal}
P.~G{\'e}rard.
\newblock Microlocal defect measures.
\newblock {\em Comm. Part. Diff. Eq.}, 16(11):1761--1794, 1991.

\bibitem{gerard1997homogenization}
P.~G{\'e}rard, P.~A. Markowich, N.~J. Mauser, and F.~Poupaud.
\newblock Homogenization limits and {W}igner transforms.
\newblock {\em Comm. Pure Appl. Math.}, 50(4):323--379, 1997.

\bibitem{golse1999vlasov}
F.~Golse and L.~Saint-Raymond.
\newblock The {V}lasov--{P}oisson system with strong magnetic field.
\newblock {\em J. math. pures appl.}, 78(8):791--817, 1999.

\bibitem{grenier1998semiclassical}
E.~Grenier.
\newblock Semiclassical limit of the nonlinear {S}chr{\"o}dinger equation in
  small time.
\newblock {\em Proc. Amer. Math. Soc.}, 126(2):523--530, 1998.

\bibitem{gui2022semiclassical}
G.~Gui and P.~Zhang.
\newblock Semiclassical limit of {G}ross--{P}itaevskii {E}quation with
  {D}irichlet {B}oundary {C}ondition.
\newblock {\em SIAM J. Math. Anal.}, 54(1):1053--1104, 2022.

\bibitem{illner1994global}
R.~Illner, P.~F. Zweifel, and H.~Lange.
\newblock Global existence, uniqueness and asymptotic behaviour of solutions of
  the {W}igner--{P}oisson and {S}chr{\"o}dinger-{P}oisson systems.
\newblock {\em Math. Meth. Appl. Sc.}, 17(5):349--376, 1994.

\bibitem{leinfelder1981schrodinger}
H.~Leinfelder and C.~G. Simader.
\newblock Schr{\"o}dinger operators with singular magnetic vector potentials.
\newblock {\em Math. Zeitschrift}, 176(1):1--19, 1981.

\bibitem{lions1993mesures}
P.-L. Lions and T.~Paul.
\newblock Sur les mesures de {W}igner.
\newblock {\em Rev. Mat. Ib.}, 9(3):553--618, 1993.

\bibitem{luhrmann2012mean}
J.~L{\"u}hrmann.
\newblock Mean-field quantum dynamics with magnetic fields.
\newblock {\em J. Math. Phys.}, 53(2):022x lüh105, 2012.

\bibitem{markowich1993classical}
P.~A. Markowich and N.~J. Mauser.
\newblock The classical limit of a self-consistent quantum-{V}lasov equation in
  3{D}.
\newblock {\em Math. Mod. Meth. Appl. Sc.}, 3(01):109--124, 1993.

\bibitem{masmoudi2001selfconsistent}
N.~Masmoudi and N.~J. Mauser.
\newblock The selfconsistent {P}auli equation.
\newblock {\em Monatshefte Math.}, 132(1):19--24, 2001.

\bibitem{masmoudi2003nonrelativistic}
N.~Masmoudi and K.~Nakanishi.
\newblock Nonrelativistic limit from {M}axwell-{K}lein-{G}ordon and
  {M}axwell-{D}irac to {P}oisson-{S}chr{\"o}dinger.
\newblock {\em Int. Math. Res. Not.}, 2003(13):697--734, 2003.

\bibitem{mauser2007convergence}
N.~J. Mauser and S.~Selberg.
\newblock Convergence of the {D}irac--{M}axwell system to the
  {V}lasov--{P}oisson system.
\newblock {\em Comm. Part. Diff. Eq.}, 32(3):503--524, 2007.

\bibitem{michelangeli2015global}
A.~Michelangeli.
\newblock Global wellposedness of the magnetic {H}artree equation with
  non-strichartz external fields.
\newblock {\em Nonlinearity}, 28(8):2743, 2015.

\bibitem{moller2023models}
J.~Möller and N.~J. Mauser.
\newblock Nonlinear {P}{D}{E} models in semi-relativistic quantum physics.
\newblock {\em To appear in CMAM}, 2023.

\bibitem{MaMo23}
J.~Möller and N.~J. Mauser.
\newblock The semiclassical limit of the {P}auli-{P}oisswell equation by the
  {W}igner method.
\newblock {\em Manuscript}, 2023.

\bibitem{nowakowski1999quantum}
M.~Nowakowski.
\newblock The quantum mechanical current of the {P}auli equation.
\newblock {\em Amer. J. Phys.}, 67(10):916--919, 1999.

\bibitem{pfaffelmoser1992global}
K.~Pfaffelmoser.
\newblock Global classical solutions of the {V}lasov-{P}oisson system in three
  dimensions for general initial data.
\newblock {\em J. Diff. Eq.}, 95(2):281--303, 1992.

\bibitem{shen1998moments}
Z.~Shen.
\newblock On moments of negative eigenvalues for the {P}auli operator.
\newblock {\em J. Diff. Eq.}, 149(2):292--327, 1998.

\bibitem{spohn2000semiclassical}
H.~Spohn.
\newblock Semiclassical limit of the {D}irac equation and spin precession.
\newblock {\em Ann. Phys.}, 282(2):420--431, 2000.

\bibitem{yang2023semi}
C.~Yang, J.~Möller, and N.~Mauser.
\newblock The semiclassical limit from the {P}auli-{P}oisswell to the
  {E}uler-{P}oisswell system by {W}{K}{B} methods.
\newblock {\em Submitted}, 2023.

\bibitem{zhang2002wigner}
P.~Zhang.
\newblock Wigner {M}easure and the {S}emiclassical {L}imit of
  {S}chr{\"o}dinger--{P}oisson {E}quations.
\newblock {\em SIAM J. Math. Anal.}, 34(3):700--718, 2002.

\bibitem{zhang2008wigner}
P.~Zhang.
\newblock {\em Wigner measure and semiclassical limits of nonlinear
  {S}chr{\"o}dinger equations}, volume~17.
\newblock Courant Lecture Notes, 2008.

\bibitem{zhang2002limit}
P.~Zhang, Y.~Zheng, and N.~J. Mauser.
\newblock The limit from the {S}chr{\"o}dinger-{P}oisson to the
  {V}lasov-{P}oisson equations with general data in one dimension.
\newblock {\em Comm. Pure Appl. Math.}, 55(5):582--632, 2002.

\end{thebibliography}
\bibliographystyle{abbrv}
\end{document}